\documentclass[a4paper,12pt]{article}

\usepackage{mathtools}
\usepackage{amsmath,amssymb,amsthm, bm}
\usepackage{enumerate}
\usepackage[noadjust]{cite}

\usepackage{setspace}

\usepackage[spanish, english]{babel}

\usepackage[ansinew]{inputenc}
\usepackage{here}
\usepackage{hyperref}

\usepackage{graphicx}

\usepackage[normalem]{ulem}

\newcounter{thm_compactness}
\newcounter{specialization_compactness}

\newcommand{\tp}{\operatorname{tp}}

\newcommand{\BB}{\mathcal{B}}
\newcommand{\CC}{\mathcal{C}}

\newcommand{\FF}{\mathcal{F}}
\newcommand{\GG}{\mathcal{G}}

\newcommand{\LL}{\mathcal{L}}
\newcommand{\MM}{\mathcal{M}}
\newcommand{\NN}{\mathcal{N}}

\newcommand{\SSS}{\mathcal{S}}
\newcommand{\UU}{\mathcal{U}}
\newcommand{\XX}{\mathcal{X}}

\newcommand{\DD}{\mathcal{D}}

\newcommand{\lemF}{F_0}
\newcommand{\lemH}{F_1}

\newcommand{\Mon}{\mathcal{U}}

\newcommand{\ar}{\rangle}
\newcommand{\al}{\langle}
\newcommand{\exR}{M \cup \{-\infty, +\infty\}}

\newcommand{\vc}{\text{vc}}
\newcommand{\VC}{\text{VC}}

\newcommand{\mytau}{\tilde{\tau}}

\newcommand{\taulim}{\lim^{\tau}}
\newcommand{\mytaulim}{\lim^{\mytau}}

\newcommand{\dcl}{\text{dcl}}

\newenvironment{claimproof}[1][\proofname]
               {
                 \proof[#1]
                 
               }
               {
                 \endproof
               }

\theoremstyle{definition}
\newtheorem{definition}{Definition}[section]
\newtheorem{remark}[definition]{Remark}
\newtheorem{example}[definition]{Example}
\newtheorem{fact}[definition]{Fact}

\theoremstyle{plain}

\newtheorem{thmx}{Theorem}
 
\newtheorem{lemma}[definition]{Lemma}
\newtheorem{proposition}[definition]{Proposition}
\newtheorem{corollary}[definition]{Corollary}
\newtheorem{claim}[definition]{Claim}

\newtheorem*{question*}{Question}

\title{Definable compactness in o-minimal structures}
\author{Pablo And\'ujar Guerrero}
\date{}

\begin{document}
\maketitle 

\noindent
{\small \emph{Mathematics Subject Classification 2020.} 03C64 (Primary); 54A05, 54D30 (Secondary). } 

\begin{abstract}
We characterize the notion of definable compactness for topological spaces definable in o-minimal structures, answering questions in~\cite{pet_stein_99} and~\cite{johnson14}. Specifically, we prove the equivalence of various definitions of definable compactness in the literature, including those in terms of definable curves, definable types and definable downward directed families of closed sets. 
\end{abstract}

\section{Introduction}\label{section:intro}

In the study of first-order topological theories various definable notions of topological compactness have been helpful tools in tame settings by isolating classes of topological objects with desirable properties. 
The first of such notions was introduced in \emph{o-minimal} theories for definable manifold spaces by Peterzil and Steinhorn~\cite{pet_stein_99}, and corresponds to the property that every definable curve converges (here curve-compactness). This property was crucial in formulating Pillay's conjecture about o-minimal definably compact groups and their relationship with compact Lie groups \cite[Conjecture 1.1]{pillay-conj}. The research that led to the solution of this conjecture provided a deeper understanding of the relationship between neostability and tame topology, in particular results in o-minimal forking were used to reach another reasonable notion of o-minimal definable compactness~\cite{pet_pillay_07}: for every definable family of closed sets with the finite intersection property there exists a finite set that intersects each set in the family (here transversal-compactness). On the other hand, Thomas~\cite{thomas12} and Walsberg~\cite{walsberg15} generalized and applied curve-compactness to study topologies arising from o-minimal definable norms and metrics respectively. In collaboration with the author~\cite{atw1}, they also explored a third notion of definable compactness within o-minimality: every downward directed definable family of nonempty closed sets has nonempty intersection (here filter-compactness). This definition has been independently studied by Johnson~\cite{johnson14} in the context of o-minimal quotient spaces, and in a general model-theoretic setting by Fornasiero~\cite{fornasiero}. The o-minimal exploration of definable compactness (through the various notions mentioned above) has yielded in particular that, in many cases, definably compact spaces are definably homeomorphic to a set with the canonical o-minimal ``Euclidean" topology (see Chapter 7 in~\cite{andujar_thesis}). Hrushovski and Loeser explored the tame topology of valued fields~\cite{hruloeser16}, including the introduction of yet another notion of definable compactness: every definable type has a limit (here type-compactness), where a limit is a point in every closed set in the type. Recently, some of these notions have also been approached in the p-adic setting by Johnson and the author~\cite{and-johnson}, and in the local o-minimal setting by Fujita~\cite{fujita23}. 

In the present paper we prove the equivalence of all the above notions of definable compactness in the setting of Hausdorff definable topological spaces (Definition~\ref{def:dts}) in o-minimal structures. We also show that, if we drop the Hausdorffness assumption, curve-compactness is strictly weaker than all the other properties. Our main result is the following (see Sections~\ref{sec:prelim-intersec} and~\ref{sec:prelim-top} for definitions).

\begin{thmx}\label{thm:intro_compactness} Fix an o-minimal structure $\MM=(M, <, \ldots)$. Let $(X,\tau)$ be a definable topological space in $\MM$. The following are equivalent. 
\begin{enumerate}[(1)]
\item \label{itm:compactness_1} Every downward directed definable family of nonempty $\tau$-closed sets has nonempty intersection (filter-compactness). 
\item \label{itm:compactness_2} Every definable type $p\in S_X(M)$ has a limit, i.e. there is a point in the intersection of every $\tau$-closed set in $p$ (type-compactness). 
\item \label{itm:compactness_2.5} Every definable family of $\tau$-closed sets that extends to a definable type in $S_X(M)$ has nonempty intersection.
\item \label{itm:compactness_3} Every consistent definable family of $\tau$-closed sets admits a finite transversal, i.e. there exists a finite set that intersects every set in the family (transversal-compactness).
\item \label{itm:compactness_5} Every definable family $\CC$ of $\tau$-closed sets with the $(m,n)$-property, where $m\geq n > \dim \cup\CC$, has a finite transversal.
\item \label{itm:compactness_4} Every definable family $\CC$ of $\tau$-closed sets with the $(m,n)$-property, where $m\geq n$ and $n$ is greater than the $\VC$-codensity of $\CC$, has a finite transversal.
\setcounter{thm_compactness}{\value{enumi}}
\end{enumerate}
Moreover all the above imply and, if $\tau$ is Hausdorff or $\MM$ has definable choice, are equivalent to:
\begin{enumerate}[(1)]
\setcounter{enumi}{\value{thm_compactness}}
\item \label{itm:compactness_6} Every definable curve in $X$ is $\tau$-completable (curve-compactness). 
\end{enumerate}
\end{thmx}

Theorem~\ref{thm:intro_compactness} and Remark~\ref{rem:Joh-question} provide a positive answer to~\cite[Question 4.14]{johnson14}, where Johnson asks whether curve-compactness and filter-compactness are equivalent for o-minimal definable manifold spaces. 

In light of Theorem~\ref{thm:intro_compactness} we may present the following definition. 

\begin{definition}\label{dfn:df-compactness}
    Let $(X,\tau)$ be a definable topological space in an o-minimal structure. We say that $(X,\tau)$ is \emph{definably compact} if it satisfies any (all) of the conditions \eqref{itm:compactness_1}-\eqref{itm:compactness_4} in Theorem~\ref{thm:intro_compactness}.
\end{definition}

We also prove that, if $\MM$ is an o-minimal expansion of the real line $(\mathbb{R},<)$, then every definable topological space in $\MM$ is definably compact if and only if it is compact in the classical sense (Corollary~\ref{cor:definably_compact_iff_compact}), which provides a positive answer (Remark~\ref{rem:pet-stein}) to part of \cite[Question 2.5]{pet_stein_99}. Furthermore, we show that definable compactness is definable uniformly in families (Proposition~\ref{prop:comp-families}). Additionally, throughout the paper we comment on other reasonable notions of definable compactness, including definitions in terms of externally definable sets (Remark~\ref{remark:peterzil_pillay}), chains (paragraph above Lemma~\ref{lemma:case_nested_family}) and nets (Remark~\ref{rem:nets}).

We prove Theorem~\ref{thm:intro_compactness} using o-minimal combinatorial and geometrical facts which are either known to hold in more general settings or can be conjectured to do so. These facts include known characterizations of o-minimal non-forking formulas, the Alon-Kleitman-Matou\v{s}ek $(p,q)$-theorem for VC classes, and two geometrical facts of independent interest about o-minimal types (Propositions~\ref{prop:existence_complete_directed_families_2.1} and~\ref{lemma:existence_complete_directed_families}), the first of which can be understood as a strong form of distal cell decomposition. Hence this paper can be seen as a road map to characterizing definable compactness in various NIP settings.

The structure of the paper is as follows. In Section~\ref{sec:prelim} we include preliminaries. In Section~\ref{sec:VC-forking} we gather the necessary literature results on Vapnik-Chervonenkis theory and on forking, extracting some easy corollaries. In Section~\ref{sec:types} we prove our two main results about o-minimal types. In Section~\ref{sec:topology} we introduce our topological framework and prove Theorem~\ref{thm:intro_compactness} through a series of propositions, as well our other results on definable compactness. 

This paper has been largely extracted from~\cite{FTT}, which includes independent proofs within o-minimality of Fact~\ref{thm:forking_definable_types} and of a version of Corollary~\ref{prop:vc_tame_transversals} where $m\geq n > \vc^*(\SSS)$ is substituted by $m\geq n > \dim \cup \SSS$, thus avoiding largely the use of forking or VC literature.

\subsection*{Acknowledgements}

The author thanks Sergei Starchenko for the idea of a characterization of definable compactness in terms of $\VC$ theory and finite transversals, which motivated much of the work in this paper. He also thanks Margaret Thomas, Pantelis Eleftheriou, Matthias Aschenbrenner, and the anonymous referee, for extensive feedback on various versions of the paper. In particular, he thanks Professor Aschenbrenner for suggesting the terminology ``$n$-consistent" and ``$n$-inconsistent". The author also thanks Antongiulio Fornasiero for sharing his unpublished notes on definable compactness~\cite{fornasiero}.  
Finally, the author is deeply indebted to Will Johnson for some very helpful conversations on the contents of Section~\ref{sec:types}, which resulted in the current proof to Proposition~\ref{lemma:existence_complete_directed_families}, and in Johnson's authoring of~\cite[Appendix A]{FTT}, a counterexample to a parameter version of Proposition~\ref{lemma:existence_complete_directed_families}.

During the research into the contents of this paper the author was supported by the Institute Henri Poincare (during the 2018 Program on Model Theory, Combinatorics and Valued Fields), the Canada Natural Sciences and Engineering Research
Council (NSERC) Discovery Grant RGPIN-06555-2018, the Graduate School and Department of Mathematics at Purdue University, the Fields Institute for Research in Mathematical Sciences (during the 2021 Thematic Program on Trends in Pure and Applied Model Theory), and the Engineering and Physical Sciences Research Council (EPSRC) Grant EP/V003291/1.  

\section{Preliminaries}\label{sec:prelim}

\subsection{Conventions}\label{sec:conv}
We fix a language $\LL=\{<,\ldots\}$ and a first-order $\LL$-structure $\MM=(M,<,\ldots)$ expanding a dense linear order without endpoints. For a set of parameters $A$ we denote by $\LL(A)$ the expansion of $\LL$ by symbols for elements in $A$. Throughout unless otherwise specified \textquotedblleft definable" means \textquotedblleft $\LL(M)$-definable in $\MM$".
All variables and parameters $x, a, u \ldots$ are $n$-tuples for some $n<\omega$. We denote the length of a variable or parameter $x$ by $|x|$. We denote ordered pairs of tuples by $\al x, y\ar$. We use $n$, $m$, $k$ and $l$ to denote natural numbers.

Unless stated otherwise, any formula we consider is in $\LL(M)$. For any formula $\varphi(x)$ and set $A\subseteq M^{|x|}$, let $\varphi(A)=\{ a\in A : \MM\models \varphi(a)\}$. For simplicity we write $\varphi(M)$ to mean $\varphi(M^{|x|})$. A (uniformly) definable family of sets is a family of the form $\{\varphi(M,b) : b\in \psi(M)\}$ for some formulas $\varphi(x,y)$ and $\psi(y)$, where we may always assume that $\varphi(x,y)\in \LL$ (i.e. a formula without parameters).   
For any two formulas $\varphi(x)$ and $\psi(x)$, we write $\varphi(x)\vdash \psi(x)$ to mean \mbox{$\MM\models \forall x (\varphi(x)\rightarrow \psi(x))$}. For sets of formulas $p(x)$ and $q(x)$ on free variables $x$, we write $p(x)\vdash q(x)$ to mean that, for every formula $\varphi(x)\in q(x)$, there is a finite subset $p'(x)\subseteq p(x)$ such that $\wedge p'(x) \vdash \varphi(x)$.
 
For a given $n$, let $\pi$ denote the projection $M^{n+1}\rightarrow M^n$ onto the first $n$ coordinates, where $n$ will often be omitted and clear from context. For a family $\SSS$ of subsets of $M^{n+1}$ let $\pi(\SSS)=\{\pi(S) :S\in\SSS\}$. 

Recall that $\MM$ is o-minimal if every definable subset of $M$ is a finite union of points and intervals with endpoints in $M\cup\{-\infty, +\infty\}$. For background in o-minimality we direct the reader to~\cite{dries98}. We will use, in particular, the existence of uniform cell decompositions~\cite[Chapter 3, Proposition 3.5]{dries98}. We use the following notation related to o-minimal cells: given two partial functions $f,g:M^n\rightarrow M\cup\{-\infty, +\infty\}$, with domains $dom(f)$ and $dom(g)$ respectively, let $(f,g)=\{\al x,t \ar : x\in dom(f)\cap dom(g),\, f(x)<t<g(x)\}$ (we relax thus the classical notation throughout by allowing that $f$ and $g$ have different domains). Whenever $\MM$ is o-minimal, we refer jointly to the order topology on $M$ and induced product topology on $M^n$ as the \emph{Euclidean topology}. Given a definable set $X\subseteq M^n$, we denote its closure in the Euclidean topology by $cl(X)$, and its frontier by $\partial(X)=cl(X)\setminus X$. We also denote the o-minimal dimension of $X$ by $\dim X$. 


\subsection{Intersecting families of sets and refinements} \label{sec:prelim-intersec}

We say that a family of sets $\SSS$ is \emph{$n$-consistent} if every subfamily of cardinality at most $n$ has nonempty intersection. A family is \emph{consistent} if it is $n$-consistent for every $n$. We say that $\SSS$ is \emph{$n$-inconsistent} if every subfamily of cardinality $n$ has empty intersection. 

A family of sets $\SSS$ has the \emph{$(p,q)$-property}, for cardinals $p\geq q>0$, if the sets in $\SSS$ are nonempty and, for every $p$ distinct sets in $\SSS$, there exists $q$ among them with nonempty intersection.
Note that $\SSS$ does not have the $(p,q)$-property if and only if it either contains the empty set or there exists a subfamily of $\SSS$ of size $p$ that is $q$-inconsistent. 

A family of sets $\SSS$ is \emph{downward directed} if, for every $F_0, F_1 \in \SSS$, there exists $F_2\in \SSS$ such that $F_2 \subseteq F_0 \cap F_1$. Equivalently if for every finite $\FF\subseteq \SSS$ there exists $F\in \SSS$ with $F\subseteq \cap\FF$. 

Given a family of sets $\SSS$ and a set $X$ let $X \cap \SSS= \SSS \cap X = \{ S\cap X : S\in\SSS\}$.
Observe that, if $\SSS$ is downward directed, then, for every set $X$, it holds that $X \cap \SSS$ is downward directed too. 

Given two families of sets $\SSS$ and $\FF$, we say that \emph{$\FF$ is a refinement of $\SSS$}, or that \emph{$\FF$ refines $\SSS$} if, for every $S\in \SSS$, there exists $F\in \FF$ with $F\subseteq S$. Observe that, if $\FF$ is a downward directed refinement of $\SSS$, then, for every finite subfamily $\GG \subseteq \SSS$, there exists some $F\in \FF$ with $F\subseteq \cap \GG$.

Given a family of sets $\SSS$ and a set $X$ we say that $X$ is a \emph{transversal of $\SSS$} if it intersects every set in $\SSS$ (i.e. $\emptyset\notin X\cap\SSS$). In this paper we are interested in the property that a definable family of sets has a finite transversal, as a weakening of the property of having nonempty intersection (i.e. having a transversal of size one). 

The following lemma will be used throughout the paper. We leave the easy proof to the reader.
\begin{lemma}\label{fact:downward_directed_family_2}
Let $\SSS$ be a downward directed family of sets and $\XX$ be a finite covering of a set $X$. If $S\cap X\neq \emptyset$ for every $S\in\SSS$, then there exists some $Y\in\XX$ such that $S\cap Y\neq \emptyset$ for every $S\in\SSS$. In particular, if $\SSS$ has a finite transversal, then $\cap\SSS \neq \emptyset$. 
\end{lemma}

\subsection{Type preliminaries}

All the types that we consider are consistent and, unless otherwise specified, complete over $M$. We denote the set of these types by $S(M)$. We denote by $S_n(M)$ the set of $n$-types in $S(M)$. We resort often and without warning to the common model-theoretic convention of identifying types with the family of sets defined by formulas in it. For a definable set $X\subseteq M^n$, we denote by $S_X(M)$ the family of all types $p\in S_n(M)$ with $X\in p$ (namely types that concentrate on $X$). We will investigate partial types which are downward directed\footnote{In the literature this property among types is also denoted \emph{$1$-compressible}.}, and the refinement relation between partial types.  

Recall that a type $p(x) \in S(M)$ is definable if, for every formula $\varphi(x,y)\in \LL$, there is another formula $\psi(y)\in \LL(M)$ such that $\psi(M)=\{b \in M^{|y|} : \psi(x,b) \in p(x) \}$. It is definable over $A\subseteq M$ if these formulas $\psi(y)$ can be chosen in $\LL(A)$.
Note that, if a type is definable, then its projection $\pi(p)$ is definable too.

Given a formula $\varphi(x)$ let $\varphi^1(x)=\varphi(x)$ and $\varphi^0(x)=\neg\varphi(x)$. Given a type $p(x)$ and a formula $\varphi(x,y)$, recall that the restriction of $p(x)$ to $\varphi(x,y)$ is the subtype $p|_\varphi (x) = \{ \varphi^i(x,b) \in p(x) : i\in \{0,1\},\, b\in M^{|y|}\}$. We denote by $p|^1_\varphi (x)$ the restriction of $p(x)$ to ``positive" instances of $\varphi(x,y)$, i.e. $p|^1_\varphi (x) = \{ \varphi(x,b) \in p(x) : b\in M^{|y|}\}$. 

\section{O-minimal VC theory and forking}\label{sec:VC-forking}


\subsection{$\VC$ theory}
The following is an ad hoc presentation of the notion of VC-codensity and related results, with applications in Sections~\ref{section:forking_dividing} and~\ref{sec:topology}. For a more standard treatment of Vapnik-Chervonenkis (VC) theory in a model-theoretic context see~\cite{vc_density}. 

A pair $(X,\SSS)$, where $X$ is a set and $\SSS$ is a family of subsets of $X$, is called a \emph{set system}. For a subfamily $\FF\subseteq \SSS$, let $BA(\FF)$ denote the collection of \emph{Boolean atoms} of $\FF$, by which we mean the family of all maximal nonempty intersections of sets in $\FF \cup \{ X \setminus S : S\in \FF\}$. The \emph{dual shatter function} of $\SSS$ is the function $\pi^*_\SSS: \omega \rightarrow \omega$ given by
\[
\pi^*_\SSS(n) = \max_{\FF\subseteq \SSS, \, |\FF|=n} |BA(\FF)|.   
\]

The \emph{$\VC$-codensity of $\SSS$}, denoted by $\vc^*(\SSS)$, is the infimum over all real numbers $r\geq 0$ such that $\pi^*_{\SSS}(n)= O(n^r)$ (that is, $\pi^*_{\SSS}(n)/n^r$ is bounded at infinity). Observe that $\vc^*(\SSS)$ is independent of the ambient set $X$, and so throughout we omit it from our terminology.
A theory $T$ is \emph{NIP} (\emph{Not the Independence Property)} if every definable family of sets in every model of $T$ has finite $\VC$-codensity. Every o-minimal theory is NIP~\cite[Chapter 5]{dries98}.

For convenience we state the Alon-Kleitman-Matou\v{s}ek $(p,q)$-theorem in terms of $\VC$-codensity. For a finer statement see~\cite[Theorem 4]{matousek04}.

\begin{fact}[Alon-Kleitman-Matou\v{s}ek $(p,q)$-theorem~\cite{matousek04}]\label{thm:vc_tranversal}
Let $p \geq q >0$ be natural numbers and let $\SSS$ be a set system such that $\vc^*(\SSS)<q$. Then there is $n<\omega$ such that, for every finite subfamily $\FF\subseteq \SSS$, if $\FF$ has the $(p,q)$-property, then it has a transversal of size at most $n$. 
\end{fact}

The following easy corollary will be used in the proof of Corollary~\ref{prop:vc_tame_transversals}.

\begin{corollary}\label{cor:pq}
Let $p \geq q >0$ be natural numbers and let $\SSS$ be a set system such that $\vc^*(\SSS)<q$. If $\SSS$ has the $(p,q)$-property, then, for every $0< q' < \omega$, there exists some natural number $p'=p'(q') \geq q'$ such that $\SSS$ has the $(p',q')$-property. In particular, $\SSS$ has the $(\omega,q')$-property for every $0< q' < \omega$.
\end{corollary}
\begin{proof}
Let $\SSS$ be as in the corollary, satisfying the $(p,q)$-property. Let $n$ be as described by Fact~\ref{thm:vc_tranversal}. For any given $q'>0$, let $p'=n(q'-1)+1$. Consider an arbitrary subfamily $\FF$ of $\SSS$ of size $p'$. By Fact~\ref{thm:vc_tranversal}, $\FF$ has a transversal $A$ of size at most $n$. By definition of $p'$, there must exist some $a\in A$ such that $|F\in \FF : a\in F|\geq q'$. It follows that $\SSS$ has the $(p',q')$-property.     
\end{proof}

The following fact is a reformulation of the main result for weakly o-minimal structures (a class which contains o-minimal structures) in~\cite{vc_density} by Aschenbrenner, Dolich, Haskell, Macpherson and Starchenko. It was previously proved for o-minimal structures by Wilkie (unpublished) and Johnson-Laskowski~\cite{john_las_10}, and for o-minimal expansions of the field of reals by Karpinski-Macintyre~\cite{kar_mac_00}.

\begin{fact}[\cite{vc_density}, Theorem 6.1]\label{cor:vc_density}
Let $\MM$ be an o-minimal structure and let $\SSS$ be a definable family of subsets of $M^n$. Then \mbox{$\vc^*(\SSS)\leq n$}.
\end{fact}

We will apply Fact~\ref{cor:vc_density} in subsequent sections through the slight improvement given by the next corollary.

\begin{corollary}\label{cor:vc_density2}
Let $\MM$ be an o-minimal structure and let $\SSS$ be a definable family of sets with $n=\dim \cup\SSS$. Then \mbox{$\vc^*(\SSS)\leq n$}.
\end{corollary}

The proof of Corollary~\ref{cor:vc_density2} follows immediately from the following lemma and o-minimal cell decomposition, the latter implying that, if $X$ is a definable set in an o-minimal structure $\MM$ with $\dim X\leq n$, then $X$ admits a finite partition into definable subsets, each of which is in definable bijection with a subset of $M^n$. 

\begin{lemma}\label{lem:vc-density-CD}
Let $\SSS$ be a set system and let $X_1, \ldots, X_m$ be sets such that $\cup\SSS \subseteq \cup_{i\leq m} X_i$. Then 
\[
\vc^*(\SSS)=\max_{1 \leq i \leq m} \vc^*(X_i \cap \SSS).
\]
\end{lemma}
\begin{proof}
First note that, for every $i\leq m$ and finite subfamily $\FF\subseteq \SSS$, $BA(X_i \cap \FF) \leq BA(\FF)+1$, meaning that $\pi_{X_i \cap \SSS}^*(n) \leq \pi_{\SSS}^*(n)+1$ for every $n$, and consequently $\vc^*(X_i \cap \SSS) \leq \vc^*(\SSS)$.

For the opposite inequality, let $\FF$ be a finite subfamily of $\SSS$. Observe that 
\[
BA(\FF) \leq BA(X_1 \cap \FF) + \cdots + BA(X_m \cap \FF).
\]
Consequently
\[
\pi_{\SSS}^*(n) \leq  \pi_{X_1 \cap \SSS}^*(n) + \cdots + \pi_{X_m \cap \SSS}^*(n)
\]
for every $n$.
It follows that, for any real number $r\geq 0$, if $\pi_{X_i \cap \SSS}^*(n)=O(n^r)$ for all $i \leq m$, then $\pi_{\SSS}^*(n)=O(n^r)$. Hence there must exist some $i \leq m$ such that $\vc^*(\SSS)\leq \vc^*(X_i \cap \SSS)$.
\end{proof}

Since throughout this paper $p$ and $q$ are employed as terminology for types, in subsequent sections we address the $(p,q)$-property in terms of $m$ and $n$, e.g. the $(m,n)$-property.

\subsection{Forking, dividing and definable types}\label{section:forking_dividing}

In this section we recall some facts about non-forking formulas in o-minimal theories, and derive some consequences which we will need in Section~\ref{sec:topology}.
This is the subject of ongoing research among NIP theories~\cite{simon15}. Throughout we fix a $|M|^+$-saturated elementary extension $\UU=(U,\ldots)$ of $\MM$.

Recall that a formula $\varphi(x,b) \in \LL(U)$ is said to \emph{$n$-divide over $A\subseteq U$}, for some $n\geq 1$, if there exists a sequence of elements $(b_i)_{i<\omega}$ in $U^{|b|}$, with $\tp(b_i/A)=\tp(b/A)$ for every $i$, such that $\{ \varphi(x,b_i) : i<\omega\}$ is $n$-inconsistent. Equivalently, $\varphi(x,b)$ is said to $n$-divide over $A$ if the family $\{\varphi(U, b') : \tp(b'/A)=\tp(b/A)\}$ does not have the $(\omega, n)$-property. A formula $\varphi(x,b)$ \emph{divides} if it $n$-divides for some $n$. Conversely, a formula $\varphi(x,b)$ does not divide over $A$ if and only if the family $\{ \varphi(U,b') : \tp(b'/A)=\tp(b/A)\}$ has the $(\omega, n)$-property for every $n$. Hence, not dividing is an intersection property. 

A formula \emph{forks over $A$} if it implies a finite disjunction of formulas that divide each over $A$. In $NTP_2$ theories (a class which includes NIP and simple theories) forking and dividing over a model are equivalent notions~\cite[Theorem $1.1$]{cher_kap_12}.

The next equivalence was proved first for o-minimal expansions of ordered fields\footnote{Dolich specifically works with ``nice" o-minimal theories, a certain class of theories which includes o-minimal expansions of ordered fields.} by Dolich~\cite{dolich04} (where he considers forking over small sets and not just models) and for unpackable $\VC$-minimal theories, a class which includes o-minimal theories, by Cotter and Starchenko~\cite{cotter_star_12}. The best generalization up to date is due to Simon and Starchenko~\cite{simon_star_14}, and applies to a large class of dp-minimal theories (for details and precise definitions of unpackable $\VC$-minimal and dp-minimal theory see~\cite{cotter_star_12} and~\cite{simon_star_14} respectively). We state the result for o-minimal theories. 

\begin{fact}\label{thm:forking_definable_types} Let $T$ be an o-minimal $\LL$-theory with monster model $\Mon$. Let $\MM\models T$ and $\varphi(x,b)\in\LL(U)$. The following are equivalent. 
\begin{enumerate}[(i)]
\item\label{itm:thm_forking_definable_types_1} $\varphi(x,b)$ does not fork (equivalently, by~\cite{cher_kap_12}, does not divide) over $M$. 
\item\label{itm:thm_forking_definable_types_2} $\varphi(x,b)$ extends to an $M$-definable type in $S_{|x|}(U)$.
\end{enumerate}
\end{fact}

In Section~\ref{sec:compactness} we will apply Fact~\ref{thm:forking_definable_types} in the form of the following corollary.


\begin{corollary}~\label{prop:vc_tame_transversals}
Let $\MM$ be an o-minimal structure and $\SSS$ be a definable family of nonempty subsets of $M^k$. If there exist natural numbers $m\geq n > \vc^*(\SSS)$ such that $\SSS$ has the $(m,n)$-property, then there exists a finite covering $\{\SSS_1, \ldots, \SSS_l\}$ of $\SSS$ by definable subfamilies such that, for every $i\leq l$, the family $\SSS_i$ extends to a definable type in $S_k(M)$.        
\end{corollary}
\begin{proof}
Let $\varphi(x,y) \in \LL$ and $\psi(y) \in \LL(M)$ be formulas such that $\SSS=\{\varphi(M, b) : b\in \psi(M)\}$. If $\SSS$ does not admit a covering as described in the corollary then, by model-theoretic compactness, there exists some $\textbf{b}\in \psi(U)$ such that $\varphi(U, \textbf{b})$ does not extend to an $M$-definable type in $S_{|x|}(U)$. On the other hand, by Corollary~\ref{cor:pq}, the family $\SSS$ has the $(\omega,n)$-property for every $n>0$, and consequently the formula $\varphi(x,\textbf{b})$ does not divide over $M$. So, by Fact~\ref{thm:forking_definable_types}, $\varphi(x,\textbf{b})$ extends to an $M$-definable type in $S_{|x|}(U)$, contradiction. 
\end{proof}

\begin{remark}\label{rem:parameter-forking}
By \cite[Theorem 3.21]{FTT} and~\cite[Corollary 5.6]{cotter_star_12}, Fact~\ref{thm:forking_definable_types} still holds if we substitute $M$ with any (small) set $A\subseteq U$.  It follows that, in Corollary~\ref{prop:vc_tame_transversals}, if $\SSS$ is $A$-definable for some $A\subseteq M$, then the finite covering $\{\SSS_1,\ldots, \SSS_l\}$ can be chosen so that each $\SSS_i$ extends to an $A$-definable type in $S_k(M)$. 
\end{remark}



\begin{remark}
There is a close relation between $(p,q)$-theorems and so-called Fractional Helly theorems (see~\cite{matousek04}), both of which branched from the classical Helly theorem. In its infinite version, this classical theorem states that every family of closed and bounded convex subsets of $\mathbb{R}^n$ that is $(n+1)$-consistent has nonempty intersection. Aschenbrenner and Fischer proved \cite[Theorem B]{aschen_fischer_11} a definable version of Helly's Theorem (i.e. for definable families of closed and bounded convex sets) in definably complete expansions of real closed fields. 

Our Theorem~\ref{thm:intro_compactness} and the arguments in~\cite[Section $3.2$]{aschen_fischer_11} allow an obvious generalization of the o-minimal part of Aschenbrenner's and Fischer's definable Helly Theorem, by asking that the sets be definably compact and closed in some (any) definable topology, instead of closed and bounded in the Euclidean sense. Perhaps more interestingly, by using Corollary~\ref{prop:vc_tame_transversals} to adapt the second proof of Theorem B in \cite{aschen_fischer_11} (the one right below Theorem 3.7), one may show that, in an o-minimal expansion $\MM$ of an ordered field, every definable family of convex subsets of $M^n$ that is $(n+1)$-consistent extends to a definable type in $S_n(M)$.
\end{remark}

\section{O-minimal types}\label{sec:types}


Throughout this section we assume that our structure $\MM$ is o-minimal. Our aim is to investigate the relationship between definable types and definable downward directed families of sets, in order to apply the results in Section~\ref{sec:topology}. Our two main results, Propositions~\ref{prop:existence_complete_directed_families_2.1} and~\ref{lemma:existence_complete_directed_families}, are of independent interest.

Proposition~\ref{prop:existence_complete_directed_families_2.1} below can be seen as a strong non-parameter form of distal cell decomposition within o-minimality (see Theorem 21(2) in~\cite{cher-sim-2}). It implies that every definable family of sets that extends to a definable type admits a refinement given by a definable downward directed family.

\begin{proposition}\label{prop:existence_complete_directed_families_2.1}
Let $p(x) \in S(M)$ be a type and $\varphi(x,y)$ be a formula. There exists another formula $\psi(x,z)$ such that $p|_\psi^{1}(x)$ is downward directed and 
\[
p|_\psi^{1} \vdash p|_\varphi.
\]
In particular, for every finite subtype $q \subseteq p|_{\varphi}$, there exists $c\in M^{|z|}$ such that $\psi(x,c)\in p(x)$ and 
$
\psi(x,c) \vdash q(x).
$
\end{proposition}

To prove the above proposition we will use the following easy lemma, whose proof we leave to the reader. 

\begin{lemma}\label{lem:conjunction}
 Let $p(x)$ be a type and $q_1(x),\ldots, q_k(x)$ be finitely many partial subtypes of $p(x)$. Suppose that, for every $i\leq k$, there exists a formula $\psi_i(x,z_i)$ such that $p|_{\psi_i}^{1}$ is downward directed and $p|_{\psi_i}^{1} \vdash q_i$.
 Then the conjunction 
 \[
 \psi(x,z_1,\ldots,z_k)=``\wedge_{i\leq k} \psi_i(x,z_i)" 
 \]
 satisfies that $p|_\psi^{1}$ is downward directed and 
 \[
 p|_\psi^{1} \vdash \bigcup_{i\leq k} q_i.
 \]
\end{lemma}

We now present the proof of the proposition. 

\begin{proof}[Proof of Proposition~\ref{prop:existence_complete_directed_families_2.1}]
We proceed by induction on $|x|$. We may assume throughout that $p(x)$ is not realized, since otherwise it suffices to have $\psi(x,z)$ be the formula $x=z$ where $|x|=|z|$. 

\textbf{Case $|x|=1$.}

By o-minimality it suffices to have $\psi(x,z_1,z_2)$, with $|z_1|=|z_2|=1$, be one of the following three formulas:  
\begin{align*}
    (z_1 < x) \wedge (x < z_2), \\
    z_1 < x, \\
    x < z_1.
\end{align*}

\textbf{Case $|x|>1$.}

Throughout let $x=(u,t)$, where $|t|=1$. Recall that $\pi(p) \in S_{|u|}(M)$ denotes the projection of the type $p$ to the first $|u|$ coordinates, i.e. $\pi(p)(u)$ is the family of all formulas $\lambda(u)$ such that $\lambda(u) \wedge (t=t)$ is in $p(x)$.

Suppose that there exists a definable partial function $f: M^{|x|-1}\rightarrow M$ whose graph is contained in $p$. By extending $f$ if necessary to a constant function outside its domain we may assume that the domain of $f$ is in fact $M^{|x|-1}$.  We may apply the induction hypothesis to the type $\pi(p)$ and formula 
\[
\varphi_f(u,y)=``\exists t ((t=f(u)) \wedge \varphi(u,t,y))",
\]
and obtain a formula $\psi_f(u,z_f)$ as described in the proposition. This allows us to construct our desired formula $\psi$ as follows:
\[
\psi(x,z_f)=\psi(u,t,z_f)=``(t=f(u)) \wedge \psi_f(u,z_f)".
\]

We show that $\psi(x,z_f)$ has the desired properties.  
Observe that, since the graph of $f$ is contained in $p$, for every $b\in M^{|y|}$ and $i\in \{0,1\}$, the formula $\varphi^i(x,b)$ belongs in $p$ if and only if $\varphi^i_f(u,b)$ belongs in $\pi(p)$. The analogous holds for $\psi(x,z_f)$ and $\psi_f(u,z_f)$. In particular, we may define $C=\{ c \in M^{|z_f|} : \psi_f(u,c) \in \pi(p)\}=\{ c \in M^{|z_f|} : \psi(x,c) \in p\}$.
Since, by induction hypothesis, the family of formulas $\{\psi_f(u,c) : c\in C\}$ is downward directed, then the same clearly holds for $p|_\psi^{1}=\{\psi(x,c) : c\in C\}$.
Moreover, for any formula of the form $\varphi^i(x,b)$ in $p(x)$, where $b\in M^{|y|}$ and $i\in \{0,1\}$, there exists $c\in C$ such that $\psi_f(u,c) \vdash \varphi_f^i(u,b)$, and so $\psi(x,c) \in p$ and $\psi(x,c) \vdash \varphi^i(x,b)$. Hence $p|_\psi^{1} \vdash p|_\varphi$.

Hence onwards we assume that there does not exists a definable partial function $f:M^{|x|-1}\rightarrow M$ whose graph is contained in $p$.

In the next paragraphs
we reduce the remaining of the proof to the case where, for every $b\in M^{|y|}$, if the formula $\varphi(x,b)$ is in $p$, then it defines a set of the form $(f_b, +\infty)$ for some partial function $M^{|x|-1}\rightarrow M \cup \{-\infty\}$. 

By o-minimal uniform cell decomposition~\cite[Chapter 3, Proposition 3.5]{dries98}, there exist finitely many formulas $\sigma_1(x,y), \ldots, \sigma_k(x,y)$ such that, for every $b\in M^{|y|}$, the family $\{ \sigma_1(M,b), \ldots, \sigma_k(M,b)\}$ is an o-minimal cell decomposition of $M^{|x|}$ compatible with $\varphi(M,b)$. Observe that
\[
\bigcup_{i\leq k} p|_{\sigma_i}^{1} \vdash p|_\varphi.
\]
By Lemma~\ref{lem:conjunction}, it suffices to pass to an arbitrary $i\leq k$ and prove the proposition for $p|_{\sigma_i}^{1}$ in place of $p|_\varphi$. Hence onwards let us assume that, for every $b\in M^{|y|}$, the formula $\varphi(x,b)$ defines a cell and, moreover, if $\varphi(x,b)\in p$, then, by assumption on $p$, this cell is of the form $(f_b, g_b)$, for $f_b$ and $g_b$ partial functions $M^{|x|-1}\rightarrow M \cup\{-\infty, +\infty\}$ with the same domain and with $f_b < g_b$. Additionally, to prove the proposition it suffices to find $\psi(x,z)$ such that $p|_\psi^{1}$ is downward directed and $p|_\psi^{1}(x)\vdash p|_\varphi^{1}(x)$. 

Recall the notation $x=(u,t)$ with $|t|=1$. Let $B=\{ b\in M^{|y|} : \varphi(x,b) \in p\}$. Let $\varphi_0(x,y)$($=\varphi_0(u,t,y)$) denote the formula $\exists s (s\leq t) \wedge \varphi(u,s,y)$, and similarly let $\varphi_1(x,y)$ be the formula $\exists s (s\geq t) \wedge \varphi(u,s,y)$. That is, for every $b\in B$, the formulas $\varphi_0(x,b)$ and $\varphi_1(x,b)$ define the sets $(f_b, +\infty)$ and $(-\infty, g_b)$ respectively. In particular, when $b\in B$, the formula $\varphi(x,b)$ is equivalent to the conjunction $\varphi_0(x,b) \wedge \varphi_1(x,b)$. So \mbox{$p|_{\varphi_0}^{1} \cup p|_{\varphi_1}^{1} \vdash p|_{\varphi}^{1}$}. By Lemma~\ref{lem:conjunction}, to prove the proposition it suffices to find formulas $\psi_0(x,z_0)$ and $\psi_1(x,z_1)$ such that, for every $j\in \{0,1\}$, the restriction $p|_{\psi_j}^{1}$ is downward directed and $p|_{\psi_j}^{1} \vdash p|_{\varphi_j}^{1}$. We prove this for $j=0$, being the remaining case analogous. For simplicity of notation we also assume that $\varphi$ is equivalent to $\varphi_0$.

Consider the formula 
\begin{align*}
\theta(u,y,y') = ``&\exists s\, \varphi(u,s,y) \wedge \exists t\, \varphi(u,t,y') \\
& \wedge \forall t (\varphi(u,t,y') \rightarrow \varphi(u,t,y))".
\end{align*}
For every $b, b'\in M^{|y|}$ note that it holds that 
\begin{equation}\label{eqn:types1}
\theta(u,b,b') \wedge \varphi(x,b') \vdash \varphi(x,b).
\end{equation}
In particular, if $b$ and $b'$ are in $B$, then $\theta(u,b,b')$ defines the set of all $u$ such that $f_b(u) \leq f_{b'}(u)$.

Recall notation $\pi(p)$ for the projection of $p$ to the first $|u|=|x|-1$ coordinates.
By induction hypothesis on the formula $\theta(u,y,y')$ and the type $\pi(p)$, there exists a formula $\xi(u,z_\xi)$ such that $\pi(p)|_{\xi}^{1}$ is downward directed and $\pi(p)|_{\xi}^{1} \vdash\pi(p)|_{\theta}$.

Finally, let $z=(z_\xi,y)$ and
\[
\psi(x,z) = \psi(u,t,z_\xi,y)= ``\xi(u,z_\xi)\wedge \varphi(x,y)".
\]
Clearly by construction $p|_\psi^{1} \vdash p|_\varphi^{1}$. We show that that $p|_\psi^{1}$ is downward directed. 

Let $D=\{d\in M^{|z_\xi|} : \xi(u,d) \in \pi(p)\}$. Note that $\xi(u,d) \wedge \varphi(x,b)$ belongs in $p$ if and only of $b\in B$ and $d\in D$. 
Let us fix $b, b' \in B$ and $d, d' \in D$. Recall that $\varphi(M, b)=(f_b, +\infty)$ and $\varphi(M, b')=(f_{b'}, +\infty)$.
Consider the formula $\zeta(u,b,b')=``\exists s \varphi(u,s,b) \wedge \exists t \varphi(u,t,b')"$, which defines the intersection of the domains of $f_b$ and $f_{b'}$. Clearly $\zeta(u,b,b') \in \pi(p)$. Observe that the sets $\theta(M,b,b')$ and $\theta(M,b',b)$ cover $\zeta(M,b,b')$, and so at least one of them belongs in $\pi(p)$. Without loss of generality we assume that $\theta(u,b,b') \in \pi(p)$. 

Let $d''\in D$ be such that $\xi(u,d'')\vdash \theta(u,b,b')$. By Equation~\eqref{eqn:types1} we have that
\[
\xi(u,d'') \wedge \varphi(x,b') \vdash \varphi(x,b).
\]
By downward directedness let $d''' \in D$ be such that 
\[
\xi(u,d''') \vdash \xi(u,d) \wedge \xi(u,d') \wedge \xi(u,d'').
\]
We conclude that
\[
\xi(u,d''') \wedge \varphi(x,b') \vdash \xi(u,d) \wedge \varphi(x,b) \wedge \xi(u,d') \wedge \varphi(x,b'),
\]
or equivalently
\[
\psi(x,d''',b') \vdash \psi(x,d,b) \wedge \psi(x,d',b').
\]
So $p|_\psi^{1}$ is downward directed.
\end{proof} 


It seems likely that Proposition~\ref{prop:existence_complete_directed_families_2.1} is also true in weakly o-minimal structures. As far as the author knows, it is open among distal dp-minimal structures. 

The following proposition shows that every definable downward directed family of nonempty sets extends to a definable type $p(x)\in S(M)$, and furthermore that $p(x)$ can be chosen so that, for some formula $\psi(x,z)$, the restriction $p|^1_\psi (x)$ is a basis (in the sense of filter basis) of cells for $p(x)$.   
We present a shorter proof than the one in~\cite[Lemma 2.7]{FTT}, applying ideas communicated to the author by Will Johnson. 

\begin{proposition}\label{lemma:existence_complete_directed_families}
Let $\varphi(x,y)$ be a formula and $B\subseteq M^{|y|}$ be such that the family $\{\varphi(x,b) : b\in B\}$ is consistent and downward directed. Then there exists a type $p(x)\in S(M)$ with $\{\varphi(x,b) : b\in B\} \subseteq p(x)$, and a formula $\psi(x,z)$ such that $p|_\psi^1$ defines a family of cells, is downward directed, and $p|_\psi^1 \vdash p$. Furthermore, if $B$ is definable, then $p(x)$ can be chosen definable too. 

In particular, for every definable downward directed family of nonempty sets $\SSS$, there exists a definable downward directed family of cells $\FF$ which refines $\SSS$ and furthermore $\FF$ generates a definable type in $S(M)$.
\end{proposition} 
\begin{proof}
We devote most of the prove to show the existence of $p(x)$ and $\psi(x,z)$ as described in the proposition except for the condition that $p|_\psi^1$ defines a family of cells. 
In the next two paragraphs we describe how, once we have these, by passing if necessary to a formula in a cell decomposition of $\psi(x,z)$ we may assume that $p|_\psi^1$ defines a family of cells, completing the proof. 

Applying uniform cell decomposition~\cite[Chapter 3, Proposition 3.5]{dries98} to the formula $\psi(x,z)$, let $\sigma_i(x,z)$, for $i\leq k$, denote formulas such that, for every $c\in M^{|z|}$, the sets $\sigma_i(M,c)$, for $i\leq k$, are a cell partition of $\psi(M,c)$. We claim that there exists some $i\leq k$ such that the family $p|_{\sigma_i}^1$ is downward directed and $p|_{\sigma_i}^1 \vdash p|_{\psi}^1$ (hence $p|_{\sigma_i}^1\vdash p$). To see this let $p|_\psi^1 = \{ \psi(x,c) : c\in C\}$ and, for every $i\leq k$, let $p|_{\sigma_i}^1=\{\sigma_i(x,c) : c\in C(i)\}$. We show that there exists $i\leq k$ such that, for every $c\in C$, there exists some $c'\in C(i)$ with $\sigma_i(x,c')\vdash \psi(x,c)$ (i.e. $p|_{\sigma_i}^1$ refines $p|_\psi^1$); hence $p|_{\sigma_i}^1 \vdash p|_{\psi}^1$ and, using the facts that $p|_\psi^1 \vdash p$ and $p|_\psi^1$ is downward directed, it is also easy to derive that $p|_{\sigma_i}^1$ is downward directed. 

Towards a contradiction suppose that, for every $i\leq k$, there exists some $c_i \in C$ such that $\sigma_i(x,c) \nvdash \psi(x,c_i)$ for every $c\in C(i)$. By downward directedness of $p|_{\psi}^1$, let $c_{k+1}\in C$ be such that $\psi(x,c_{k+1})\vdash \wedge_{i\leq k} \psi(x,c_i)$. It follows that $\sigma_i(x,c) \nvdash \psi(x,c_{k+1})$ for every $i\leq k$ and $c\in C(i)$. However this contradicts the facts that $\psi(x,c_{k+1})\in p(x)$ and $\vdash \psi(x,c_{k+1}) \leftrightarrow \vee_{i\leq k} \sigma_i(x,c_{k+1})$, which imply that there exists some $i\leq k$ with $\sigma_i(x,c_{k+1}) \in p(x)$ (i.e. $c_{k+1}\in C(i)$) and $\sigma_i(x,c_{k+1})\vdash \psi(x,c_{k+1})$.

We now begin the prove of the existence of a type $p(x) \in S(M)$ extending $\{\varphi(x,b) : b\in B\}$ and a formula $\psi(x,z)$ satisfying that $p|_\psi^1$ is downward directed and $p|_\psi^1 \vdash p$ (i.e. $p|_\psi^1$ is a basis for $p$). We prove the case where $B$ is definable. In the general case the same proof applies by considering throughout, instead of definable families of sets, subfamilies of fibers of definable sets in general. To make the presentation more succinct, we work explicitly with set notation rather than formulas. 

We introduce some useful terminology. For a definable family of nonempty sets $\FF$, let $d(\FF)$ denote the smallest $n\geq 0$ such that, for every set $F\in \FF$, there exists $G\in \FF$ with $G\subseteq F$ and $\dim(G)=n$. Let $c(\FF)$ denote the smallest $m\geq 1$ such that, for every set $F\in \FF$, there exists $G\in \FF$ with $G\subseteq F$ such that $G$ has exactly $m$ definably connected components.

Let $\SSS=\{\varphi(M,b) : b\in B\}$. Recall that a family of sets $\FF$ is a refinement of $\SSS$ if, for every $S\in \SSS$, there exists $F\in \FF$ with $F\subseteq S$. Let $\mathcal{DR}$ denote the collection of all definable downward directed refinements of $\SSS$ which do not contain the empty set. Throughout we fix $n=\min\{d(\FF) : \FF\in\mathcal{DR}\}$ and $m = \min \{c(\FF) : \FF\in \mathcal{DR},\, d(\FF)=n\}$. We also fix $\FF \in \mathcal{DR}$ with $d(\FF)=n$ and $c(\FF)=m$. We show that $\FF$ generates a (clearly definable) type in $S_{|x|}(M)$. 

Towards a contradiction we assume that $\FF$ does not generate a type in $S_{|x|}(M)$, meaning that there exists a definable set $X\subseteq M^{|x|}$ satisfying that, for every $F\in \FF$, $F\cap X\neq \emptyset$ and $F\setminus X\neq \emptyset$. Let us fix some $\lemF\in \FF$ with $\dim \lemF=n$. 

Consider the boundary of $\lemF \cap X$ in $\lemF$, i.e. the set $Z=\lemF \cap (\partial(\lemF\cap X) \cup \partial(\lemF\setminus X))$. Since $\dim \lemF=n$, by o-minimality we have that $\dim Z < n$. It follows that $\FF \cap Z$ is a downward directed refinement of $\SSS$ composed of sets of dimension lower than $n$. By definition of $n$, there must exist a set $\lemH\in \FF$ with $\lemH\cap Z=\emptyset$. Now let $\GG=\{F \cap X : F \in \FF,\, F\subseteq \lemF \cap \lemH\}$. By downward directedness of $\FF$ and definition of $X$ the definable family $\GG$ is a downward directed refinement of $\SSS$ that does not contain the empty set (i.e. $\GG\in \mathcal{DR}$). By definition of $n$ it follows that $d(\GG)=n$. We show that $c(\GG)<m$, contradicting the definition of $m$.

We show that, for every $F\in \FF$ with $F\subseteq \lemF\cap \lemH$, the intersection $F\cap X \in \GG$ has strictly less definably connected components than $F$. In particular, this implies that, for every set $F\in \FF$ with $F\subseteq \lemF\cap \lemH$, if $G\in \FF$ is a subset of $F$ with exactly $m$ definably connected components, then $G\cap X \in \GG$ has less than $m$ definably connected components, and so $c(\GG) < m$ as desired.  
 
Let $Y$ denote the interior of $\lemF\cap X$ in $\lemF$, i.e. $Y=\lemF\setminus cl(\lemF\setminus X)$. Let $C \subseteq \lemF$ be a definably connected set. If $C\cap Z=\emptyset$, then by definition of $Z$ clearly $C$ must be a subset of either $Y$ or $\lemF\setminus (Y\cup Z)$. Since $Y\subseteq \lemF \cap X \subseteq Y\cup Z$, then $C$ must be a subset of either $F_0 \cap X$ or $\lemF\setminus X$. Now let us fix a set $F\in \FF$ with $F\subseteq \lemF \cap \lemH$. Since $\lemH\cap Z=\emptyset$ we have that $F\cap Z = \emptyset$, and so every definably connected component $C$ of $F$ is a subset of either $F_0 \cap X$ or $F_0\setminus X$. Finally recall that, by definition of $X$, the sets $F \cap X$ and $F \setminus X$ are both nonempty. Consequently we conclude that the set $F \cap X$ (as well as $F\setminus X$) has a positive number of definably connected components that is lesser than the number of definably connected components of $F$.  
\end{proof}

In Proposition~\ref{lemma:existence_complete_directed_families}, whenever $B$ is definable, one may wonder if $p(x)$ can always be chosen definable over the same parameters as $B$. This was proved to be false in general by Johnson in~\cite[Appendix B]{FTT}. Nevertheless, by~\cite[Proposition 2.17]{FTT} it does hold that every definable downward directed family $\{\varphi(x,b) : b\in B\}$ extends to a type in $S_{|x|}(M)$ definable over the same parameters as $B$. This can also be proved using Corollary~\ref{prop:vc_tame_transversals} and Remark~\ref{rem:parameter-forking}. For a similar result see~\cite[Lemma
4.2.18]{hruloeser16}.
 
\begin{remark}
Observe that Propositions~\ref{prop:existence_complete_directed_families_2.1} and~\ref{lemma:existence_complete_directed_families} together yield a strong density result for types $p(x)$ satisfying that there is a formula $\varphi(x,y)$ such that $p|_\varphi$ is downward directed and $p|_\varphi\vdash p$, namely types which have a basis (in the sense of filter basis) given by their restriction to a single formula. This is discussed in~\cite[Remarks 2.13 and 2.22]{FTT}. In any o-minimal structure every $1$-type is of this kind (it is either realized or has a basis of open intervals). On the other hand, it was shown in~\cite[Corollary 32]{atw1} that, in an o-minimal expansion of an ordered group, every definable type of this kind contains at least one set of dimension at most $2$ (and of dimension at most $1$ in o-minimal expansions of ordered fields). Using the Marker-Steinhorn Theorem~\cite[Theorem 2.1]{mark_stein_94} one derives that, in any o-minimal expansion of the group of reals, there are $n$-types that do not have a basis given by their restriction to a single formula, for every $n>2$ ($n>1$ in o-minimal expansions of the field of reals). 
\end{remark}

\section{O-minimal definable compactness}\label{sec:topology}


\subsection{Topological preliminaries} \label{sec:prelim-top}

We introduce definable (explicitly in the sense of Flum and Ziegler~\cite{FZ}) topological spaces and various related definitions. 

\begin{definition}\label{def:dts}
A \emph{definable topological space} $(X,\tau)$, with $X\subseteq M^n$, is a topological space such that there exists a definable family of subsets of $X$ which is a basis for $\tau$.

\end{definition}  

Any definable set in an o-minimal structure with its induced Euclidean topology is a definable topological space. For other examples within o-minimality, see the definable manifold spaces studied by Pillay~\cite{pillay88} and van den Dries~\cite[Chapter 10]{dries98}, the definable Euclidean quotient spaces of van den Dries~\cite[Chapter 10]{dries98} and Johnson~\cite{johnson14}, the definable normed spaces of Thomas~\cite{thomas12}, and the definable metric spaces of Walsberg~\cite{walsberg15}. See moreover the author's doctoral dissertation~\cite{andujar_thesis} for an exhaustive exploration of o-minimal definable topological spaces. 
For a foundational treatment of definable tame topology generalizing o-minimality see the work of Pillay~\cite{pillay87}. For an exploration of dp-minimal tame topology see the more recent work of Simon and Walsberg~\cite{sim-wals19}, and related independent work of Dolich and Goodrick~\cite{dolich-good22}.

Onwards we contextualize topological notions related to a given topological space $(X,\tau)$ by adding the prefix $\tau$, e.g. $\tau$-open, $\tau$-closure etc. We recall some standard definitions. 

\begin{definition}
Let $(X,\tau)$ be a definable topological space. A \emph{definable curve in $X$} is a definable map $\gamma:(a,b)\rightarrow X$, for some $-\infty\leq a < b \leq +\infty$. We say that it \emph{$\tau$-converges} to $x\in X$ (i.e. $x$ is a \emph{$\tau$-limit of $\gamma$)} as $t\rightarrow a$ if, for every $\tau$-neighborhood $A$ of $x$, there exists $t_A\in (a,b)$ such that $\gamma(s)\in A$ whenever $s\in (a,t_A)$. The notion of $\tau$-convergence as $t\rightarrow b$ is defined analogously. We denote by $\taulim_{t\rightarrow a} \gamma(t)$ (respectively $\taulim_{t\rightarrow b} \gamma(t)$) the set of $\tau$-limit points of $x$ as $t\rightarrow a$ (respectively $t\rightarrow b$). 

We say that $\gamma$ is \emph{$\tau$-completable} if it $\tau$-converges as $t\rightarrow a$ and as $t\rightarrow b$.
\end{definition}

Given a definable topological space $(X,\tau)$ and a set $Y\subseteq X$ we denote the $\tau$-closure of $Y$ by $cl_\tau(Y)$.
It is easy to check that a $\tau$-limit of a definable curve $\gamma:(a,b)\rightarrow Y\subseteq X$ is always contained in $cl_\tau(Y)$. Furthermore, if $\tau$ is Hausdorff, then the sets $\taulim_{t\rightarrow a} \gamma(t)$ and $\taulim_{t\rightarrow b} \gamma(t)$ are always either empty or a singleton, and in the latter case we abuse terminology by identifying them with their single point. We will use these facts in Section~\ref{sec:compactness} without explanation. To erase ambiguity, at times we also use side convergence notation $t\rightarrow a^+$ and $t\rightarrow b^-$ (e.g. $\taulim_{t \rightarrow a^+} \gamma(t)$), with the standard meaning.

The following definition is borrowed from~\cite{hruloeser16}.

\begin{definition}
Let $(X,\tau)$ be a definable topological space and $p$ be a (possibly partial) type with $X\in p$. We say that $x\in X$ is a \emph{$\tau$-limit\footnote{Fornasiero~\cite{fornasiero}, as well as Thomas, Walsberg and the author~\cite{atw1}, use the word ``specialization" (borrowed from real algebraic geometry) to refer to limits of types. Here we use instead the terminology from Hrushovski and Loeser~\cite[Chapter 4]{hruloeser16}.} of $p$} if $x$ is contained in the $\tau$-closure of every subset of $X$ in $p$. If $p(x)\in S_X(M)$, then this is equivalent to saying that $x$ in contained in every $\tau$-closed set in $p$.
\end{definition}

We now present various definitions extracted from the literature (for references see Section~\ref{section:intro}) which seek to capture the notion of definable compactness. We mostly maintain consistency with~\cite{and-johnson} in the names. (In particular we avoid using the adjective ``definable" in our terminology to enable an easier read.) A more general approach to definable compactness, including more definitions than the ones in this paper, can be found in unpublished work of Fornasiero~\cite{fornasiero}. 

\begin{definition}\label{dfn:compact}
Let $(X,\tau)$ be a definable topological space. Then $(X,\tau)$ is: 
\begin{enumerate}[(1)]
\item \emph{curve-compact} if every definable curve in $X$ is $\tau$-completable. 
\item \emph{filter-compact} if every downward directed definable family of nonempty $\tau$-closed subsets of $X$ has nonempty intersection,
\item \emph{type-compact} if every definable type $p(x) \in S_X(M)$ has a $\tau$-limit in $X$.
\item \emph{transversal-compact} if every consistent definable family of $\tau$-closed subsets of $X$ has a finite transversal. 
\end{enumerate}
\end{definition}  

The equivalence between curve-compactness and filter-compactness was proved for definable topological spaces in o-minimal expansions of ordered fields in~\cite[Corollary 44]{atw1}. In this paper we present a deeper characterization in the general o-minimal setting. 

\subsection{Characterizing definable compactness} \label{sec:compactness}

In this section we prove our results on definable compactness for definable topological spaces in o-minimal structures.
Throughout we assume that our underlying structure $\MM$ is o-minimal.

We devote most of the section to proving the characterization of definable compactness given by Theorem~\ref{thm:intro_compactness}, which we divide into three propositions. Proposition~\ref{prop:pre_specialization_compactness} provides the equivalence \eqref{itm:compactness_2}$\Leftrightarrow$\eqref{itm:compactness_2.5} in the theorem. In Proposition~\ref{prop:compactnes_transversals} we prove, using results from previous sections, the equivalence between~\eqref{itm:compactness_1}, \eqref{itm:compactness_2.5}, \eqref{itm:compactness_3}, \eqref{itm:compactness_5} and~\eqref{itm:compactness_4}. Finally, in Proposition~\ref{prop:curve_compact_iff_directed_compact} we prove the implication~\eqref{itm:compactness_1}$\Rightarrow$\eqref{itm:compactness_6}, and the reverse implication when $\tau$ is Hausdorff or $\MM$ has definable choice. We follow it with an example (Example~\ref{example:space-compact}) showing that the implication~\eqref{itm:compactness_1}$\Rightarrow$\eqref{itm:compactness_6} is strict in general. Throughout we also discuss other notions of definable compactness, and end the section with two additional results: definable compactness is equivalent to classical compactness in o-minimal expansions of $(\mathbb{R},<)$ (Corollary~\ref{cor:definably_compact_iff_compact}), and definable compactness is definable in families (Proposition~\ref{prop:comp-families}).

The equivalence $\eqref{itm:specialization_compactness_i}\Leftrightarrow\eqref{itm:specialization_compactness_ii}$ in Proposition~\ref{prop:pre_specialization_compactness} below corresponds to the equivalence $\eqref{itm:compactness_2}\Leftrightarrow\eqref{itm:compactness_2.5}$ in Theorem~\ref{thm:intro_compactness}. Note that the proof of this equivalence does not use o-minimality. Hence this characterization of type-compactness holds in any model-theoretic structure. 
Furthermore, the equivalence of type-compactness with classical compactness always holds whenever the underlying structure $\MM$ satisfies that every type in $S(M)$ is definable, as we point out in Remark~\ref{rem:type-comp} below.

\begin{proposition}\label{prop:pre_specialization_compactness}
Let $(X,\tau)$ be a definable topological space. The following are equivalent. 
\begin{enumerate}[(1)]
\item \label{itm:specialization_compactness_i} $(X,\tau)$ is type-compact.
\item \label{itm:specialization_compactness_ii} Every definable family of $\tau$-closed sets that extends to a definable type in $S_X(M)$ has nonempty intersection.
\setcounter{specialization_compactness}{\value{enumi}}
\end{enumerate}
If $\MM$ expands $(\mathbb{R},<)$, then \eqref{itm:specialization_compactness_i} and \eqref{itm:specialization_compactness_ii} are also equivalent to:
\begin{enumerate}[(1)]
\setcounter{enumi}{\value{specialization_compactness}}
\item \label{itm:specialization_compactness_iii} $(X,\tau)$ is compact. 
\end{enumerate}
\end{proposition}
\begin{proof}
To prove \eqref{itm:specialization_compactness_i}$\Rightarrow$\eqref{itm:specialization_compactness_ii}, suppose that $(X,\tau)$ is type-compact and let $\CC$ be a definable family of $\tau$-closed sets that extends to a definable type $p \in S_X(M)$. Let $x\in X$ be a $\tau$-limit of $p$. Then clearly $x\in \cap\CC$. 

The key element to the rest of the proof is the fact that every closed set in a topological space is an intersection of basic closed sets. 

To prove \eqref{itm:specialization_compactness_ii}$\Rightarrow$\eqref{itm:specialization_compactness_i}, let $p\in S_X(M)$ be a definable type. Let $\BB$ denote a definable basis (of opens) for the topology $\tau$. Now let $\CC$ denote the definable family of basic $\tau$-closed sets in $p$, i.e. the family of sets $C$ in $p$ of the form $X\setminus B$ for some $B\in \BB$. If \eqref{itm:specialization_compactness_ii} holds, then there exists some $x\in X$ with $x\in \cap \CC$. In this case it follows that $x$ is a $\tau$-limit of $p$.

Finally, suppose that $\MM$ expands $(\mathbb{R},<)$. Clearly, if $(X,\tau)$ is compact, then it is type-compact. Conversely, suppose that $(X,\tau)$ is type-compact and let $\CC$ be a consistent family of $\tau$-closed sets. The intersection $\cap\CC$ can be rewritten as an intersection of basic closed sets. In particular, we may assume that $\CC$ contains only definable sets. Now, by the Marker-Steinhorn Theorem~\cite[Theorem 2.1]{mark_stein_94}, every type over $M=\mathbb{R}$ is definable. Consequently $\CC$ extends to a definable type $p \in S_X(M)$. Let $x$ be a $\tau$-limit of $p$, then $x\in \cap \CC$. So $(X,\tau)$ is compact. 
\end{proof}

\begin{remark}\label{rem:type-comp}
Note that the equivalence between type-compactness and classical topological compactness shown in Proposition~\ref{prop:pre_specialization_compactness} holds in every structure satisfying that all types are definable. For example it remains true in the field of p-adic numbers $(\mathbb{Q}_p,+,\cdot)$, as observed in~\cite[Theorem 8.15]{and-johnson}.   

More specifically, if $\varphi(x,y)$ defines a basis $\BB$ for the topology $\tau$, i.e. $\BB=\{\varphi(M,b) : b\in M^{|y|}\}$, then to have the equivalence between type-compactness and classical compactness it suffices to have that every maximal consistent subfamily of $\{ X\setminus B : B\in \BB\}$ is definable, which occurs in particular whenever every $\varphi$-type (i.e. restrictions of types in $S_{|x|}(M)$ to $\varphi(x,y)$) is definable. Observe that the latter always holds whenever $\varphi(x,y)$ is stable although, as already noted in the proof of~\cite[Proposition 1.2]{pillay87}, every infinite $T_1$ topological space that has a basis defined by a stable formula must be discrete, and consequently not compact.   
\end{remark}




Proposition~\ref{prop:compactnes_transversals} below corresponds to the equivalence between~\eqref{itm:compactness_1}, \eqref{itm:compactness_2.5}, \eqref{itm:compactness_3}, \eqref{itm:compactness_5} and~\eqref{itm:compactness_4} in Theorem~\ref{thm:intro_compactness}. Its proof relies on Proposition~\ref{prop:existence_complete_directed_families_2.1} and Corollaries~\ref{cor:vc_density2} and~\ref{prop:vc_tame_transversals}.

\begin{proposition}\label{prop:compactnes_transversals}
Let $(X,\tau)$ be a definable topological space. The following are equivalent. 
\begin{enumerate}[(1)]
\item \label{itm:compactness_transversals_1} $(X,\tau)$ is filter-compact. 
\item \label{itm:compactness_transversals_1.5} Every definable family of $\tau$-closed sets that extends to a definable type in $S_X(M)$ has nonempty intersection.
\item \label{itm:compactness_transversals_4} Every definable family $\CC$ of $\tau$-closed sets with the $(m,n)$-property, where $m\geq n > \vc^*(\SSS)$, has a finite transversal.
\item \label{itm:compactness_transversals_3} Every definable family $\CC$ of $\tau$-closed sets with the $(m,n)$-property, where $m\geq n >\dim \cup\CC$, has a finite transversal.
\item \label{itm:compactness_transversals_2} $(X,\tau)$ is transversal-compact.
\end{enumerate}
\end{proposition}
\begin{proof}
Note that, if a downward directed family of sets has a finite transversal, then, by Lemma~\ref{fact:downward_directed_family_2}, it has nonempty intersection. Hence \eqref{itm:compactness_transversals_4}, \eqref{itm:compactness_transversals_3} and~\eqref{itm:compactness_transversals_2} each imply~\eqref{itm:compactness_transversals_1}. We prove~\eqref{itm:compactness_transversals_1}$\Rightarrow$\eqref{itm:compactness_transversals_1.5} and~\eqref{itm:compactness_transversals_1.5}$\Rightarrow$\eqref{itm:compactness_transversals_4}. Observe that implication~\eqref{itm:compactness_transversals_4}$\Rightarrow$\eqref{itm:compactness_transversals_3} follows from Corollary~\ref{cor:vc_density2}, and implication~\eqref{itm:compactness_transversals_3}$\Rightarrow$\eqref{itm:compactness_transversals_2} is trivial, completing the proof.  

\textbf{Proof of~\eqref{itm:compactness_transversals_1}$\Rightarrow$\eqref{itm:compactness_transversals_1.5}.}

Suppose that $(X,\tau)$ is filter-compact and let $\CC=\{\varphi(M,b) : b\in B\}$ be a definable family of $\tau$-closed sets that extends to a definable type $p(x) \in S_X(M)$. Let $\psi(x,z)$ be as given by Proposition~\ref{prop:existence_complete_directed_families_2.1} for $\varphi(x,y)$ and $p(x)$. Let $\FF=\{ \psi(X,c) : \psi(x,c)\in p(x),\, c\in M^{|z|}\}$. By Proposition~\ref{prop:existence_complete_directed_families_2.1}, $\FF$ is a definable downward directed family of subsets of $X$ which refines $\CC$. Let $\mathcal{D}=\{ cl_\tau(F) : F \in \FF\}$. Clearly $\mathcal{D}$ is a definable downward directed family of $\tau$-closed sets, so by filter-compactness there exists $a\in \cap\DD$. Moreover observe that, since the sets in $\CC$ are closed, then $\DD$ is still a refinement of $\CC$, implying that $\cap\DD \subseteq \cap \CC$, and so $a\in \cap \CC$.

\textbf{Proof of~\eqref{itm:compactness_transversals_1.5}$\Rightarrow$\eqref{itm:compactness_transversals_4}.}

Let $\CC$ be a definable family of $\tau$-closed subsets of $X$ with the $(m,n)$-property, where $m\geq n > \vc^*(\CC)$. By Corollary~\ref{prop:vc_tame_transversals} there exists a finite covering $\{\CC_1, \ldots, \CC_k\}$ of $\CC$ by definable subfamilies, each of which extends to a definable type in $S_X(M)$. If property~\eqref{itm:compactness_transversals_1.5} holds, then, for each $i\leq k$, there exists some $a_i \in \cap \CC_i$ in $X$. The family $\{a_1,\ldots, a_k\}$ is clearly a transversal of $\CC$.
\end{proof}

\begin{remark}
We remark that, although omitted from the proof above, the implication \eqref{itm:compactness_transversals_1.5}$\Rightarrow$\eqref{itm:compactness_transversals_1} in Proposition~\ref{prop:compactnes_transversals} (i.e. by Proposition~\ref{prop:pre_specialization_compactness} the implication \emph{type-compactness $\Rightarrow$ filter-compactness}) can be shown to follow easily from Proposition~\ref{lemma:existence_complete_directed_families}. In fact we claim that this implication, as well as \eqref{itm:compactness_transversals_1.5}$\Rightarrow$\eqref{itm:compactness_transversals_2}, hold in a more general dp-minimal setting by~\cite[Theorem 5]{simon_star_14} (see the discussion above Fact~\ref{thm:forking_definable_types}). Additionally, the equivalence $\eqref{itm:compactness_transversals_4}\Leftrightarrow\eqref{itm:compactness_transversals_2}$ holds in all in NIP structures by recent work of Kaplan~\cite[Corollary 4.9]{kaplan22}.
\end{remark}

\begin{remark}\label{remark:peterzil_pillay}
It was shown in~\cite[Theorem 2.1]{pet_stein_99} that a definable set with the o-minimal Euclidean topology is curve-compact if and only if it is closed and bounded. In \cite[Theorem 2.1]{pet_pillay_07} Peterzil and Pillay extracted from~\cite{dolich04} the following. Suppose that our o-minimal structure $\MM$ has definable choice (e.g. expands an ordered group). Let $\Mon=(U,\ldots)$ be a monster model and $\varphi(x,b)$ be a formula in $\LL(U)$ such that $\varphi(U, b)$ is closed and bounded (in the Euclidean topology in $U^{|x|}$). If the family $\{ \varphi(U,b') : \tp(b'/M)=\tp(b/M)\}$ is consistent, then $\varphi(U,b)$ has a point in $M^{|x|}$. Using a straightforward model-theoretic compactness argument they derive from this that every closed and bounded Euclidean space is transversal-compact~\cite[Corollary 2.2 $(i)$]{pet_pillay_07}.

Let $(X,\tau)$ be a definable topological space (in $\MM$), whose definition in $\Mon$ we denote by $(X(\Mon), \tau(\Mon))$. The property that every formula $\varphi(x,b)\in \LL(U)$, satisfying that $\varphi(U,b)$ is $\tau(\Mon)$-closed and the family $\{ \varphi(U,b') : \tp(b'/M)=\tp(b/M)\}$ is consistent, satisfies that $\varphi(M,b)\neq \emptyset$, is labelled Dolich's property in~\cite{fornasiero}. As mentioned in the previous paragraph, this property implies transversal-compactness (without any assumption on $\MM$), and furthermore one may show, using recent work of Kaplan~\cite[Theorem 1.5]{kaplan22}, that the converse implication (transversal-compactness$\Rightarrow$Dolich's property) holds in all NIP structures. 

Theorem~\ref{thm:intro_compactness} completes the characterization of closed and bounded definable sets with the Euclidean topology. Furthermore, it generalizes Peterzil's and Pillay's \cite{pet_pillay_07} aforementioned result in three ways. First, we drop the assumption of having definable choice in $\MM$. Second, we weaken the consistency assumption to having an appropriate $(n,m)$-property (in their work they actually observe that it suffices to have $k$-consistency for some $k$ in terms of $|x|$ and $|b|$). Third, we establish, by means of the equivalence with transversal-compactness mentioned in the paragraph above, the relationship between Dolich's property and the other compactness notions in the full generality of any o-minimal definable topological space. 
\end{remark}

We now prove the connection within o-minimality between filter-compactness and curve-compactness stated in Theorem~\ref{thm:intro_compactness}. That is, that filter-compactness implies curve-compactness, and that both notions are equivalent when the topology is Hausdorff or when the underlying o-minimal structure has definable choice. This is Proposition~\ref{prop:curve_compact_iff_directed_compact}. We follow the proposition with an example of a non-Hausdorff topological space definable in the dense linear order without endpoints $(M,<)$ that is curve-compact but not filter-compact.  

The next lemma allows us to apply definable choice in certain instances even when the underlying structure $\MM$ may not have the property.

\begin{lemma}[Definable choice in compact Hausdorff spaces]\label{lemma:choice_compact_spaces}
Let $C$ be a definable nonempty $\tau$-closed set in a curve-compact Hausdorff definable topological space $(X,\tau)$. Let $A\subseteq M$ be such that $\tau$ and $C$ are $A$-definable. Then there exists a point $x\in C \cap \dcl(A)$, where $\dcl(A)$ denotes the set of finite tuples of elements in the definable closure of $A$. 

Consequently, for every $A$-definable family $\{\varphi(M,b) : b\in B\}$ of nonempty subsets of $X$, there exists an $A$-definable choice function $h:B\rightarrow X$ such that $h(b)\in cl_\tau(\varphi(M,b))$ for every $b\in B$. 

\end{lemma}
\begin{proof}
We prove the first paragraph of the lemma. The uniform result is derived in the usual way by the use of first-order logic compactness.

For this proof we adopt the convention of the one point Euclidean space $M^0=\{\bm{0}\}$. In particular, any projection $M^{k}\rightarrow M^0$ is simply the constant function $\bm{0}$, and any relation $E\subseteq M^0\times M^{k}$ is definable if and only if its projection to $M^{k}$ is.  

Let $C$, $(X,\tau)$ and $A$ be as in the lemma, with $X\subseteq M^m$. Let $n\leq m$ be such that there exists an $A$-definable function $f:D\subseteq M^n\rightarrow C$, for $D$ a nonempty set. If $n$ can be chosen to be zero, then the lemma follows. We prove that this is the case by backwards induction on $n$. 

Note that $n$ can always be chosen equal to $m$, by letting $f$ be the identity on $C$. Consider a positive $n\leq m$. For every $x\in M^{n-1}$, let $D_x$ denote the fiber $\{t\in M : \al x,t\ar \in D\}$. For each $x\in\pi(D)$, let $s_x=\sup D_x$, and consider the $A$-definable set $F=\{x\in \pi(D): s_x\in D_x\}$. 

If $F\neq \emptyset$, then let $g$ be the map $x\mapsto f(s_x):F\rightarrow C$. 
If $F=\emptyset$, then let $g$ be the map $x\mapsto \taulim_{t\rightarrow s^-_x} f(x,t): \pi(D)\rightarrow C$ which, by curve-compactness and Hausdorffness, is well defined. In both cases $g$ is an $A$-definable nonempty partial function $M^{n-1}\rightarrow C$. 
\end{proof}

\begin{remark}
Let $A\subseteq M$ and $\CC$ be an $A$-definable family of nonempty $\tau$-closed sets in a curve-compact Hausdorff $A$-definable topological space $(X,\tau)$. Lemma~\ref{lemma:choice_compact_spaces} implies that, if $\CC$ has a finite transversal, then it also has one of the same size in $\dcl(A)$. To prove this it suffices to note that, for every $k\geq 1$, the set of $k$-tuples of points corresponding to a transversal of $\CC$ is $A$-definable and closed in the product topology, which can easily be shown to be $A$-definable and curve-compact. 

It follows that, whenever $\MM$ has definable choice or $\tau$ is Hausdorff, the finite transversals in Theorem~\ref{thm:intro_compactness} (statements~\eqref{itm:compactness_3}, \eqref{itm:compactness_5} and~\eqref{itm:compactness_4}) can always be assumed to be definable over the same parameters as the family of closed sets $\CC$ and topology $\tau$.
\end{remark}

\begin{proposition}\label{prop:curve_compact_iff_directed_compact}
Let $(X,\tau)$ be a definable topological space. If $(X,\tau)$ is filter-compact, then it is curve-compact. 

Suppose that either $\tau$ is Hausdorff or $\MM$ has definable choice. Then $(X,\tau)$ is filter-compact if and only if it is curve-compact. 
\end{proposition}

We prove the left to right direction through a short lemma. 
\begin{lemma}\label{lemma:curve_compactness_implies_compactness}
Let $(X,\tau)$ be a filter-compact definable topological space. Then $(X,\tau)$ is curve-compact. 
\end{lemma}
\begin{proof}
Let $\gamma:(a,b)\rightarrow X$ be a definable curve in $X$. Consider the definable family of $\tau$-closed nested sets $\CC_\gamma=\{ cl_\tau \gamma[(a,t)] : a<t<b\}$. By filter-compactness, there exists $x\in\cap\CC_\gamma$. By o-minimality, observe that $\gamma$ satisfies that it $\tau$-converges to $x$ as $t\rightarrow a$. Similarly one shows that $\gamma$ also $\tau$-converges as $t\rightarrow b$.    
\end{proof}

We now prove a simpler case of the left to right implication in Proposition~\ref{prop:curve_compact_iff_directed_compact}. We do so by implicitly introducing a weakening of filter-compactness corresponding to the property that every definable family of nonempty closed sets that is nested has nonempty intersection (say chain-compactness). We show that, when $\MM$ has definable choice or the underlying topology is Hausdorff, curve-compactness implies chain-compactness (the reverse implication always holds within o-minimality by the proof of Lemma~\ref{lemma:curve_compactness_implies_compactness}). On the other hand, Example~\ref{example:space-compact} describes a (non-Hausdorff) definable topological space in $(M,<)$ that is curve-compact but not chain-compact. It is unclear whether chain-compactness is equivalent to definable compactness (Definition~\ref{dfn:df-compactness}) in the general setting of o-minimal definable topological spaces.  

\begin{lemma}\label{lemma:case_nested_family}
Let $(X,\tau)$ be a definable topological space. Suppose that either $\tau$ is Hausdorff or $\MM$ has definable choice. 
Let $\CC$ be a nested definable family of nonempty $\tau$-closed subsets of $X$. If $(X,\tau)$ is curve-compact, then $\cap\CC\neq \emptyset$. 
\end{lemma}
\begin{proof}
Let $(X,\tau)$ and $\CC=\{\varphi(M,b) : b\in B\}$, with $B\subseteq M^n$, be as in the lemma. We assume that $(X,\tau)$ is curve-compact and show that $\cap \CC\neq \emptyset$. We proceed by induction on $n$. 
 
\textbf{Case $n=1$.}

Consider the definable total preorder $\preceq$ in $B$ given by $b\preceq c$ if and only if $\varphi(M,b) \subseteq \varphi(M,c)$. If $B$ has a minimum $\bm{b}$ with respect to $\preceq$, then $\varphi(M, \bm{b})\subseteq \varphi(M,c)$ for every $c\in B$, and the result follows. 
We suppose that $(B,\preceq)$ does not have a minimum and consider the definable nested family of (necessarily infinite) sets $\{ (-\infty, b)_{\preceq} : b\in B\}$, where $(-\infty, b)_\preceq=\{ c\in B : c \preceq b\}$ for every $b\in B$. 
Now let $a=\sup \{ \inf (-\infty, b)_{\preceq} : b\in B\}$, where the infimum and supremum are taken in $\exR$ with respect to the order $<$ in $M$. We show that one of the families $\{(a,t): t>a\}$ or $\{(t,a): t<a\}$ is a refinement of  $\{ (-\infty, b)_{\preceq} : b\in B\}$. We prove the case where $a\in M$, being the case where $a\in \{-\infty, +\infty\}$ similar but more straightforward.

Towards a contradiction suppose that $\{ (-\infty, b)_{\preceq} : b\in B\}$ does not have a refinement as described. Then, by o-minimality, there exists some $b_1 \in B$ and $t_1>a$ such that $(-\infty, b_1)_{\preceq} \cap (a,t_1) = \emptyset$, and similarly there exists $b_2 \in B$ and $t_2<a$ such that $(-\infty, b_2)_{\preceq} \cap (t_2,a) = \emptyset$. Additionally, by definition of $a$, there exists $b_3\in B$ such that $\inf (-\infty, b_3)_{\preceq} > t_2$. Finally let $b_4\in B$ be such that $a\notin (-\infty, b_4)_{\preceq}$. Now let $i \leq 4$ be such that $b_i$ is the minimum with respect to $\preceq$ in $\{b_1, b_2, b_3, b_4\}$. Then observe that $\inf (-\infty, b_i)_{\preceq} \geq t_1$, contradicting the definition of $a$.

Hence onwards we assume that the family of intervals $\{(a,t): t>a\}$ is a refinement of $\{ (-\infty, b)_{\preceq} : b\in B\}$, being the case where the refinement is given by the family $\{(t,a): t<a\}$ analogous. This means that, for every $b\in B$, there exists an element $t(b)>a$ such that $\varphi(M,c)\subset \varphi(M,b)$ for every $c\in (a,t(b))$.

Since either $\MM$ has definable choice or $\tau$ is Hausdorff there exists, by Lemma~\ref{lemma:choice_compact_spaces}, a definable function $f:B\rightarrow \cup \CC$ satisfying that $f(b)\in \varphi(M,b)$ for every $b\in B$. By the above paragraph it follows that, for every $b\in B$, if $c\in (a, t(b))$, then $f(c) \in \varphi(M,b)$. Let $d>a$ be such that $(a,d)\subseteq B$ and $\gamma$ be the restriction of $f$ to $(a,d)$. We derive that, for every $C\in \CC$, it holds that $\taulim_{t\rightarrow a} \gamma(t) \subseteq C$. Since $(X,\tau)$ is curve-compact we conclude that $\cap \CC\neq \emptyset$.

\textbf{Case $n>1$.}

For every $x\in\pi(B)$, let $\CC_x$ denote the family $\{\varphi(M,x,t) : t\in B_x\}$, where $B_x=\{t\in M : \al x, t \ar \in B$\}, and set $C(x):=\cap \CC_x$. By the case $n=1$ the definable family of $\tau$-closed sets $\DD=\{C(x) : x\in\pi(B)\}$ does not contain the empty set. Clearly $\cap \DD=\cap \CC$. We observe that the family $\DD$ is nested and the result follows from the induction hypothesis. 

Given $x ,y\in \pi(B)$, if, for every $C\in \CC_x$, there is $C'\in\CC_y$ with $C'\subseteq C$, then $\cap \CC_y  \subseteq \cap \CC_x$. Otherwise, there is $C\in\CC_x$ such that, for every $C'\in\CC_y$, it holds that $C\subseteq C'$, in which case $\cap \CC_x\subseteq C \subseteq \cap \CC_y$.     
\end{proof}

We may now prove Proposition~\ref{prop:curve_compact_iff_directed_compact}.

\begin{proof}[Proof of Proposition~\ref{prop:curve_compact_iff_directed_compact}]
By Lemma~\ref{lemma:curve_compactness_implies_compactness} we must only prove the right to left implication in the second paragraph.
Let $(X,\tau)$, with $X\subseteq M^m$, be a curve-compact definable topological space, where either $\tau$ is Hausdorff or otherwise $\MM$ has definable choice. Let $\CC$ be a definable downward directed family of nonempty subsets of $X$, not necessarily $\tau$-closed. We show that $\bigcap\{cl_\tau(C) : C\in\CC\}\neq \emptyset$. 

We proceed by induction on $n=\min\{\dim C : C\in\CC\}$. Applying Proposition~\ref{lemma:existence_complete_directed_families}, after passing to a refinement of $\CC$ if necessary, we may assume that $\CC$ is a downward directed family of cells that generates a type in $S_m(M)$.
If $n=0$, then there exists a finite set in $\CC$ and so (see Lemma~\ref{fact:downward_directed_family_2}) $\cap\CC\neq \emptyset$. Hence onwards we assume that $n>0$. We begin by proving the case $n=m$. Hence suppose that every $C\in \CC$ is an open cell $C=(f_C,g_C)$, for functions $f_C ,g_C:\pi(C)\rightarrow \exR$ with $f_C < g_C$. Onwards recall the notation fixed in the last paragraph of Section~\ref{sec:conv}. For every $C\in\CC$, consider the definable set $D(C)=\bigcap \{cl_\tau(f_C,g_{C'}) : C'\in \CC\}$. 

\begin{claim}\label{claim:D(C)}
For every $C\in \CC$ it holds that $D(C)\neq\emptyset$.
\end{claim} 
\begin{claimproof}
Let us fix $C=(f,g)$ and, for every $x\in\pi(C)$, let $c_x$ denote a point in $\taulim_{t\rightarrow f(x)^+} \al x,t\ar$. If $\tau$ is Hausdorff then, for every $x\in \pi(C)$, there is a unique choice for $c_x$, otherwise we use definable choice to pick $c_x$ definably in $x$. The definable set $C^0=\{c_x : x\in\pi(C)\}$ has dimension at most $\dim(\pi(C))=\dim(C)-1=n-1$. For every $C'=(f',g')\in\CC$, since $C\cap C' \neq\emptyset$, the definable set $\{x\in \pi(C)\cap \pi(C'): f(x)<g'(x)\} \supseteq \pi(C\cap C')$ is nonempty, and so $C^0 \cap cl_\tau(f,g')\neq \emptyset$. 

Note that, because $\CC$ is downward directed, the definable family of nonempty sets $\{C^0 \cap cl_\tau(f, g_{C'}) : C'\in\CC\}$ is downward directed. Since $\dim C^0 \leq n-1$, by inductive hypothesis there is a point that belongs in the $\tau$-closure of $C^0 \cap cl_\tau(f, g_{C'})$ --in particular in $cl_\tau(f, g_{C'})$-- for all $C'\in\CC$. Hence $D(C)\neq \emptyset$.
\end{claimproof}

Note that, for every $C\in\CC$, it holds that $D(C)\subseteq cl_\tau(C)$. We now show that the definable family of nonempty (by Claim~\ref{claim:D(C)}) $\tau$-closed sets $\{D(C) : C\in\CC\}$ is nested. Then, by Lemma~\ref{lemma:case_nested_family}, $\bigcap \{ D(C) : C\in\CC\}\neq \emptyset$, and thus $\bigcap\{ cl_\tau(C) : C\in\CC\}\neq \emptyset$. 

Let us fix $C_1=(f_1,g_1)$ and $C_2=(f_2,g_2)$ in $\CC$. We may partition $B=\pi(C_1)\cap \pi(C_2)$ into the definable sets 
\[
B_1=\{x\in B : f_1(x)\leq f_2(x)\} \text{ and } B_2=\{x\in B : f_1(x)>f_2(x)\}. 
\]
Since $\CC$ is a basis (i.e. a downward directed generating family) for a type in $S_m(M)$, there exists some $i\in\{1,2\}$ and $C\in \CC$ such that $\pi(C)\subseteq B_i$. Without loss of generality suppose that $i=1$, and let us fix $C_3\in\CC$ with $\pi(C_3)\subseteq B_1$. For an arbitrary set $C=(f,g)\in \CC$, let $C'=(f',g')\in\CC$ be such that $C'\subseteq C \cap C_3$. Then, clearly $(f_2,g')\subseteq (f_1,g')\subseteq (f_1,g)$. It follows that $D(C_2)\subseteq D(C_1)$. This completes the proof of the case $n=m>0$. 

Finally, we describe how the proof in the case $0<n<m$ can be obtained by adapting the arguments above. Fix $\hat{C}\in\CC$ with $\dim \hat{C}=n<m$ and a projection $\hat{\pi}:M^m\rightarrow M^n$ such that $\hat{\pi}|_{\hat{C}}:\hat{C}\rightarrow \hat{\pi}(\hat{C})$ is a bijection. By passing to a refining subfamily of $\CC$ if necessary we may assume that every set in $\CC$ is contained in $\hat{C}$. By definition of $n$ it follows that the definable downward directed family $\hat{\pi}(\CC)=\{\hat{\pi}(C) : C\in\CC\}$ contains only open cells in $M^n$. Note moreover that this family is a basis for a type in $S_n(M)$.


Set $h:=(\hat{\pi}|_{\hat{C}})^{-1}$ and, for every $C\in \CC$, let $\hat{\pi}(C)=(\hat{f}_C,\hat{g}_C)$. Then the proof in the case $n=m$ can be applied by letting $D(C)$ be the intersection $\bigcap \{cl_\tau h[(\hat{f}_C,\hat{g}_{C'})] : C'\in \CC\}$, and letting $C^0$ be a set given by points $c_x$ in $\taulim_{t\rightarrow \hat{f}(x)^+} h(x,t)$ chosen definably in $x\in\pi(\hat{\pi}(C))$.
  
\end{proof}

The following is an example of a non-Hausdorff definable topological space in the unbounded dense linear order $(M,<)$ that is curve-compact but not definably compact. In particular, the space admits a definable nested family of nonempty closed sets with empty intersection (see the comments above Lemma~\ref{lemma:case_nested_family}).

\begin{example}\label{example:space-compact}
Let $X=\{\al x,y\ar\in M^2 : y<x\}$. 
Consider the family $\BB$ of subsets of $X$ of the form
\begin{align*}
A(x',x'',x''',y',y'',y''')=&\{\al x,y\ar : y<y', y<x\} \cup \\
&\{\al x,y\ar : y''<y<y''' \wedge (y<x<y''' \vee x'<x<x'' \vee x'''<x)\}
\end{align*}
definable uniformly in $y'<y''<y'''<x'<x''<x'''$.

\begin{figure}[h!]
\centering
\includegraphics[scale=0.65]{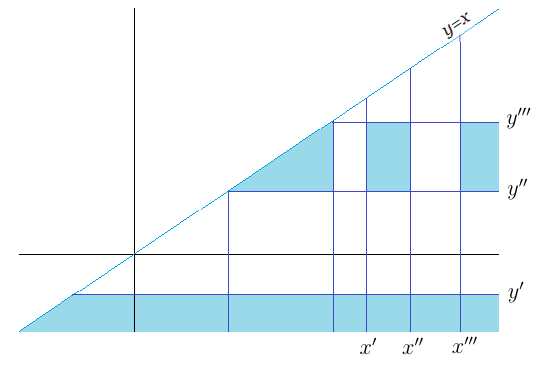}
\caption{\footnotesize Depicting (in blue) the set $A(x',x'',x''',y',y'',y''')$. \normalsize}
\end{figure}

Given any $A_0=A(x_0',x_0'',x_0''',y_0',y_0'',y_0''')$ and $A_1=A(x_1',x_1'',x_1''',y_1',y_1'',y_1''')$ in $\BB$, and any $\al x,y\ar\in A_1\cap A_2$, since every set in $\BB$ is open in the Euclidean topology, we may find $y''<y<y'''<x'<x<x''$ such that the box $(x',x'')\times (y'',y''')$ is a subset of $A_1\cap A_2$. Let $y'<\min\{y'',y'_0,y'_1\}$ and $x'''>\max\{x'',x'''_0,x'''_1\}$. Then $\al x,y\ar\in A(x',x'',x''',y',y'',y''')\subseteq A_1 \cap A_2$. Hence the family $\BB$ is a definable basis for a topology $\mytau$ on $X$. 

This topology is $T_1$, i.e. every singleton is $\mytau$-closed. For every $y\in M$, $\mytaulim_{t\rightarrow y^+} \al t, y\ar = \mytaulim_{t\rightarrow +\infty} \al t, y\ar = (M\times\{y\})\cap X$, and, for every $x\in M$, $\mytaulim_{t\rightarrow x^-} \al x, t\ar=(M\times\{x\})\cap X$ and $\mytaulim_{t\rightarrow -\infty} \al x, t\ar=X$. In particular, $\mytau$ is not Hausdorff. 

Now suppose that $\MM=(M,<)$. By quantifier elimination we know that in this structure any definable partial map $M\rightarrow M$ is piecewise either constant or the identity. Let $\gamma=(\gamma_0,\gamma_1):(a,b)\rightarrow X$ be an injective definable curve in $X$, where $\gamma_0$ and $\gamma_1$ are the projections to the first and second coordinates respectively. Let $I=(a,c)\subseteq (a,b)$ be an interval where $\gamma_0$ and $\gamma_1$ are either constant or the identity. Since the graph of the identity is disjoint from $X$ and $\gamma$ is injective it must be that $\gamma_i$ is constant and $\gamma_{1-i}$ is the identity on $I$ for some $i\in\{0,1\}$. 

Suppose that $\gamma_1|_I$ is constant with value $y$. Then, by the observations made above about the topology $\mytau$, the curve $\gamma$ satisfies that it $\mytau$-converges as $t\rightarrow a$ to either $\al a , y \ar$ (if $y<a$) or $(M\times\{y\}) \cap X$ (if $a=y$). On the other hand, if $\gamma_0|_I$ has a constant value $x$, then $\gamma$ $\mytau$-converges as $t\rightarrow a$ to either $\al x , a \ar$ (if $a>-\infty$) or the whole space $X$ (if $a=-\infty$). Analyzing the limit as $t\rightarrow b$ similarly allows us to conclude that $\gamma$ is $\mytau$-completable. Hence the space $(X,\mytau)$ is curve-compact.  

Meanwhile, the definable nested family of $\mytau$-closed sets $\{ X \cap (M \times [b,+\infty)) : b\in M\}$ has empty intersection. In particular, $(X,\mytau)$ is not definably compact. 
\end{example}

\begin{remark}\label{rem:Joh-question}
In \cite[Question 4.14]{johnson14} Johnson asks whether curve-compactness and filter-compactness are equivalent for o-minimal definable manifold spaces \cite[Chatper 10, Section 1]{dries98}. While these spaces are not necessarily Hausdorff observe that, because they admit a covering by finitely many Hausdorff open subspaces, every definable curve in them converges to only finitely many points. It follows that the proofs of Lemmas~\ref{lemma:choice_compact_spaces} and~\ref{lemma:case_nested_family} and Proposition~\ref{prop:curve_compact_iff_directed_compact} adapt to these spaces without any assumption that they are Hausdorff or that $\MM$ has definable choice. Hence we can answer Johnson's question in the affirmative. In fact, by Theorem~\ref{thm:intro_compactness}, every definable manifold space is curve-compact if and only if it is definably compact. 
\end{remark}

In the next remark we relate definable compactness to definable nets. 

\begin{remark} \label{rem:nets}
In~\cite[Section 6]{atw1}, Thomas, Walsberg and the author introduce the notion of \emph{definable net} $\gamma:(B,\preceq)\rightarrow (X,\tau)$ to mean a definable map from a definable directed set $(B, \preceq)$ into a definable topological space $(X,\tau)$. Recall that \emph{(Kelley) subnet} of $\gamma$ is a net $\gamma':(B',\preceq') \rightarrow (X,\tau)$ such that $\gamma'= \gamma \circ \mu$, where $\mu:B'\rightarrow B$ satisfies that, for every $b\in B$, there exists $c \in B'$ satisfying that $\mu(d) \succeq b$ for every $d \succeq c$. We say that such a net $\gamma'$ is definable if $(B',\preceq')$ and $\mu$ are definable.

Classically, a topological space is compact if and only if every net in it has a converging subnet. Following the classical proof of this result one may show that, in any model-theoretic structure (regardless of the axiom of o-minimality), filter-compactness implies that every definable net has a definable converging subnet (say net-compactness). The reverse implication follows whenever the structure has definable choice. In the o-minimal case one may easily show that net-compactness implies curve-compactness and so, by Theorem~\ref{thm:intro_compactness}, net-compactness implies definable compactness whenever the topology is Hausdorff or $\MM$ has definable choice. See~\cite[Corollary 44]{atw1} for a detailed proof of the equivalence between net-compactness, curve-compactness and filter-compactness in o-minimal expansions of ordered groups. Furthermore, one may show that Example~\ref{example:space-compact} above is not net-compact. The author is unaware of whether net-compactness implies definable compactness always within o-minimality. 

In general topology there are further characterizations of compactness, such as the property that every net has a cluster point, which are not addressed in this paper.  
\end{remark}

In light of Definition~\ref{dfn:df-compactness}, the following corollary is a direct consequence of Proposition~\ref{prop:pre_specialization_compactness}.

\begin{corollary}\label{cor:definably_compact_iff_compact}
Suppose that $\MM$ is an o-minimal expansion of $(\mathbb{R},<)$ and let $(X,\tau)$ be a definable topological space. Then $(X,\tau)$ is definably compact if and only if it is compact. 
\end{corollary}

\begin{remark}\label{rem:pet-stein}
Part of~\cite[Question 2.5]{pet_stein_99} asks whether curve-compactness is equivalent to classical compactness among definable manifold spaces in o-minimal expansions of $(\mathbb{R},<)$.  Theorem~\ref{thm:intro_compactness}, Remark~\ref{rem:Joh-question} and Corollary~\ref{cor:definably_compact_iff_compact} above provide a positive answer to this question. The other part of~\cite[Question 2.5]{pet_stein_99} asks whether curve-compactness for o-minimal definable manifold spaces is maintained after passing to an o-minimal expansion. Recall that Example~\ref{example:space-compact} describes a (non-Hausdorff) definable topological space $(X,\mytau)$ in an arbitrary dense linear order $(M,<)$ which is curve-compact but not definably compact. The linear order $(M,<)$ could be chosen to have an o-minimal expansion $\NN$ with definable choice, in which case, by Theorem~\ref{thm:intro_compactness}, in $\NN$ the space $(X,\mytau)$ would lose the property of curve-compactness. Hence the above question has a negative answer when directed at all (non-Hausdorff) definable topological spaces.
On the other hand, by Corollary~\ref{cor:definably_compact_iff_compact}, definable compactness is always maintained after passing to an expansion in o-minimal structures over $(\mathbb{R}, <)$. It remains open however whether the same is true in arbitrary o-minimal structures.   
\end{remark}

Notice that all the characterizations of definable compactness in Theorem~\ref{thm:intro_compactness} are upfront expressible with infinitely many sentences (possibly with parameters) in the language of $\MM$, i.e. you have to check all relevant definable families of closed sets or all definable curves. Recall that, by~\cite[Theorem 2.1]{pet_stein_99} (or alternatively \cite[Proposition 3.10]{johnson14}) and Theorem~\ref{thm:intro_compactness}, a definable set with the Euclidean topology is definably compact if and only if it is closed and bounded. Being closed and bounded is expressible by a single formula (in the same parameters used to define the set). Furthermore, given a definable family of sets with the Euclidean topology, the subfamily of those that are closed and bounded is definable.  

We generalize this last observation to all definably compact spaces definable in o-minimal structures, showing along the way that definable compactness is expressible with a single formula (in the same parameters that define the topological space). We prove this for type-compactness and apply Theorem~\ref{thm:intro_compactness}. We use the following fact regarding strict pro-definability of the space of definable types from~\cite{cubides-ye-hils}, which can also be proved using Propositions~\ref{prop:existence_complete_directed_families_2.1} and~\ref{lemma:existence_complete_directed_families}, together with the Marker-Steinhorn Theorem~\cite[Theorem 2.1]{mark_stein_94}. Recall that throughout this section our structure $\MM$ is o-minimal. 

\begin{fact}[\cite{cubides-ye-hils} Theorems 2.4.8 and 3.2.1]\label{fact:strict-PDT}
For every formula $\varphi(x,y)$, there exists a formula $\psi(y,z)$ and a definable (over $\emptyset$) set $Z \subseteq M^{|z|}$ with the following properties. For every $d\in Z$, there exists a definable type $p(x)\in S_{|x|}(M)$ such that
\begin{equation}\label{eqn:Strict-PDT}
\{ b \in M^{|y|} : \varphi(x,b)\in p(x) \} = \psi(M,d),
\end{equation}
and vice versa for every definable type $p\in S_{|x|}(M)$ there exists some $d\in Z$ satisfying equation~\eqref{eqn:Strict-PDT}. 
\end{fact}




\begin{proposition}\label{prop:comp-families}
Let $\{ (X_c, \tau_c) : c\in C\}$, be a definable family of topological spaces, i.e. there exists a definable set $B\subseteq M^{n+m}$ and some formula $\sigma(x, u, v)$, with $|u|=n$ and $|v|=m$, such that $C$ is the projection of $B$ to the last $m$ coordinates and, for every $c\in C$, the family $\{ \sigma(M, b, c) : b\in B_c\}$, where $B_c=\{ b \in M^n : \al b,c \ar \in B\}$, is a basis for $\tau_c$. Then the subfamily of all definably compact spaces is definable, i.e. there exists a definable set $D\subseteq C$ such that $(X_c,\tau_c)$ is definably compact if and only if $c\in D$. 

In particular, given any definable family of subsets of a definable topological space, the subfamily of those that are definably compact is definable.  
\end{proposition}
\begin{proof}
Let $\varphi(x, u, v)$ be a formula such that, for every $c \in C \subseteq M^{|v|}$, the family $\{ \varphi(M, b, c) : b\in B_c \}$ is a basis of closed sets for $\tau_c$.
Recall that, as observed in the proof of Proposition~\ref{prop:pre_specialization_compactness}, a type has a limit if and only if the intersection of all \textbf{basic} closed sets in it is nonempty.

Let $\psi(u,v,z)$ and $Z \subseteq M^{|z|}$ be as given by Fact~\ref{fact:strict-PDT} with $y=(u,v)$. 
It follows that, for any fixed $c\in C$, the space $(X_c, \tau_c)$ is type-compact if and only if either $X_c$ is empty or the following holds:
\[
\forall z\in Z \, (\exists x \, \forall u \, ( \psi(u,c,z) \wedge u\in B_c \rightarrow \varphi(x,u,c))).
\]
Since the set $Z$ is definable, this completes the proof.
\end{proof}

In the case where $\MM$ is an o-minimal expansion of an ordered field, the completeness of the theory of tame pairs \cite{dries_lewenberg_95} can be used to circumvent some of the proofs in this paper. In particular Thomas, Walsberg and the author~\cite[Corollary 47]{atw1} use it to prove the equivalence between curve-compactness and type-compactness among definable topological spaces, drawing inspiration from Walsberg's previous proof~\cite[Proposition 6.6]{walsberg15} of the equivalence between curve-compactness and sequential compactness among definable metric spaces in the case where $\MM$ expands the field of reals.


\bibliographystyle{alpha}
\bibliography{mybib_types_transversals_compactness}

\end{document}